\documentclass{article}
\usepackage{mypaper}

\usepackage[english]{babel}
\usepackage[utf8]{inputenc}
\usepackage[T1]{fontenc} 

\usepackage{graphicx}
\usepackage{epstopdf}

\usepackage[mode=buildnew]{standalone}
\usepackage{tikz}
\usetikzlibrary{shapes,arrows,positioning,external,calc,math}
\usepackage{pgfplots}

\usepackage{url}
\usepackage{verbatim}
\usepackage{multirow, booktabs, adjustbox} 
\usepackage{algorithm}
\usepackage{algorithmic}
\usepackage{enumerate}

\usepackage{amssymb}
\makeatletter
\@ifpackageloaded{amsmath}{}{\usepackage[nosumlimits]{amsmath}}
\makeatother
\usepackage{fixmath} 
\usepackage{xfrac} 
\usepackage[makeroom]{cancel} 
\usepackage{dsfont}
\usepackage{mathrsfs}
\usepackage{mathtools}
\usepackage{centernot}

\newcommand{\norm}[2][]{{#1\|}{#2}{#1\|}}
\newcommand{\inprod}[3][]{{#1\langle}{#2},{#3}{#1\rangle}}

\DeclareMathOperator*{\argmin}{arg\,min}

\DeclareMathOperator{\Exp}{\mathbb{E}}

\DeclareMathOperator{\Id}{Id}

\newcommand{\R}{\mathbb{R}}

\newcommand{\N}{\mathbb{N}}

\usepackage{amsthm}

\newtheorem{prop}{Proposition}[section]
\numberwithin{prop}{section}

\newtheorem{ass}{Assumption}[section]
\numberwithin{ass}{section}

\newtheorem{lem}{Lemma}[section]
\numberwithin{lem}{section}

\newtheorem{thm}{Theorem}[section]
\numberwithin{thm}{section}

\numberwithin{cor}{section}

\numberwithin{defin}{section}

\numberwithin{rem}{section}

\newcommand{\propref}[1]{Proposition~\ref{#1}}
\newcommand{\assref}[1]{Assumption~\ref{#1}}
\newcommand{\lemref}[1]{Lemma~\ref{#1}}
\newcommand{\thmref}[1]{Theorem~\ref{#1}}

\makeatletter
\@ifpackageloaded{cleveref}{
	\crefname{prop}{Proposition}{Propositions}
	\crefname{ass}{Assumption}{Assumptions}
	\crefname{lem}{Lemma}{Lemmas}
	\crefname{thm}{Theorem}{Theorems}
	\crefname{cor}{Corollary}{Corollaries}
	\crefname{defin}{Definition}{Definitions}
	\crefname{rem}{Remark}{Remarks}
}{}
\makeatother

\newcommand{\algref}[1]{Algorithm~\ref{#1}}
\newcommand{\secref}[1]{Section~\ref{#1}}

\newenvironment{bigmatrix}
{
\setlength\arraycolsep{0.7em}

\begin{bmatrix}
}
{
\end{bmatrix}
}

\newcommand{\keywords}[1]{}

\title{Cocoercivity, Smoothness and Bias in Variance-Reduced Stochastic Gradient Methods}

\author{
	Martin Morin\\
	Dept. of Automatic Control\\
	Lund University \\
	\texttt{martin.morin@control.lth.se}
	\And
	Pontus Giselsson\\
	Dept. of Automatic Control\\
	Lund University \\
	\texttt{pontus.giselsson@control.lth.se}
}

\begin{document}
\maketitle

\begin{abstract}
	With the purpose of examining biased updates in variance-reduced stochastic gradient methods, we introduce SVAG, a SAG/SAGA-like method with adjustable bias.
	SVAG is analyzed in a cocoercive root-finding setting, a setting which yields the same results as in the usual smooth convex optimization setting for the ordinary proximal-gradient method.
	We show that the same is not true for SVAG when biased updates are used.
	The step-size requirements for when the operators are gradients are significantly less restrictive compared to when they are not.
	This highlights the need to not rely solely on cocoercivity when analyzing variance-reduced methods meant for optimization.
	Our analysis either match or improve on previously known convergence conditions for SAG and SAGA.
	However, in the biased cases they still do not correspond well with practical experiences and we therefore examine the effect of bias numerically on a set of classification problems.
	The choice of bias seem to primarily affect the early stages of convergence and in most cases the differences vanish in the later stages of convergence.
	However, the effect of the bias choice is still significant in a couple of cases.
\end{abstract}

\section{Introduction}
Variance-reduced stochastic gradient (VR-SG) methods is a family of iterative optimization algorithms that combine the low per-iteration computational cost of the ordinary stochastic gradient descent and the attractive convergence properties of gradient descent.
Just as ordinary stochastic gradient descent, VR-SG methods solve smooth optimization problems on finite sum form,
\begin{equation}\label{eq:optprob}
	\min_{x\in\R^N} \tfrac{1}{n}\sum_{i=1}^n f_i(x)
\end{equation}
where,  for all $i\in\{1,\dots,n\}$, $f_i:\R^N \to \R$ is a convex function that is $L$-smooth, i.e., $f_i$ is differentiable with $L$-Lipschitz continuous gradient.
These types of problems are common in model fitting, supervised learning, and empirical risk minimization which, together with the nice convergence properties of VR-SG methods, has lead to a great amount of research on VR-SG methods and the development of several different variants, e.g., \cite{le_roux_stochastic_2012,schmidt_minimizing_2017,defazio_saga_2014,johnson_accelerating_2013,xiao_proximal_2014,hofmann_variance_2015,konecny_semi-stochastic_2017,defazioFinitoFasterPermutable2014,mairalOptimizationFirstorderSurrogate2013,allen-zhuKatyushaFirstDirect2017,nguyenSARAHNovelMethod2017,kovalevDonJumpHoops2019,tangAcceleratingSGDUsing2019,shalev-shwartzStochasticDualCoordinate2013,hanzelySEGAVarianceReduction2018}.

Broadly speaking, VR-SG methods form a stochastic estimate of the objective gradient by combining one or a few newly evaluated terms of the gradient with all previously evaluated terms.
Classic examples of this can be seen in the SAG \cite{le_roux_stochastic_2012,schmidt_minimizing_2017} and SAGA \cite{defazio_saga_2014} algorithms.
Given some initial iterates $x^0,y_1^0,\dots,y_n^0\in\R^N$ and step-size $\lambda > 0$, SAGA samples $i^k$ uniformly from $\{1,\dots,n\}$ and then updates the iterates as
\begin{align*}
	x^{k+1} &= x^k - \lambda \left(\nabla f_{i^k}(x^k) - y_{i^k}^k  + \tfrac{1}{n}\sum_{j=1}^n y_j^k \right), \\
	y_{i^k}^{k+1} &= \nabla f_{i^k}(x^k), \\
	y_j^{k+1} &= y_j^k  \quad\text{for all}\quad j \neq i^k,
\end{align*}
for $k\in\{0,1,\dots\}$.
The update of $x^{k+1}$ is said to be unbiased since the expected value of $x^{k+1}$ at iteration $k$ is equal to an ordinary gradient descent update.
This is in contrast to the biased SAG, which is identical to SAGA except that the update of $x^{k+1}$ is
\begin{align*}
	x^{k+1} &= x^k - \lambda \left(\tfrac{1}{n}\big(\nabla f_{i^k}(x^k) - y_{i^k}^k \big)  + \tfrac{1}{n}\sum_{j=1}^n y_j^k \right)
\end{align*}
and the expected value of $x^{k+1}$ now includes a term containing the old gradients $\frac{1}{n}\sum_{i=1}^n y_i^k$.
Although SAG shows that unbiasedness is not essential for the convergence of VR-SG methods, the effects of this bias are unclear.
The majority of VR-SG methods are unbiased but existing works have not established any clear advantage of either the biased SAG or the unbiased SAGA.
This paper will examine the effect of bias and its interplay with different problem assumptions for SAG/SAGA-like methods.

\subsection{Problem and Algorithm}\label{sec:probalg}

Instead of solving \eqref{eq:optprob} directly, we consider a closely related but more general root-finding problem.
Throughout the paper, we consider the Euclidean space $\R^N$ and the problem of finding $x\in\R^N$ such that
\begin{equation}\label{eq:zerprob}
	0 = Rx \coloneqq \tfrac{1}{n}\sum_{i = 1}^n R_ix
\end{equation}
where $R_i : \R^N \to \R^N$ is $\frac{1}{L}$-cocoercive---see \secref{sec:prel}---for all $i\in\{1,\dots,n\}$.
Since $L$-smoothness of a convex function is equivalent to $\frac{1}{L}$-cocoercivity of the gradient \cite[Corollary 18.17]{bauschke_convex_2017}, the smooth optimization problem in \eqref{eq:optprob} can be recovered by setting $R_i = \nabla f_i$ for all $i\in\{1,\dots,n\}$ in \eqref{eq:zerprob}.
Problem \eqref{eq:zerprob} is also interesting in its own right with it and the closely related fixed point problem of finding $x\in\R^N$ such that $x = (\Id - \alpha R)x$ where $\alpha \in (0,2L^{-1})$ both having applications in for instance feasibility and non-linear signal recovery problems, see \cite{combettesSolvingMonotoneInclusions2004,combettesFixedPointFramework2021,combettesSolvingCompositeFixed2020} and the references therein.
To solve this problem, we present the \emph{Stochastic Variance Adjusted Gradient} (SVAG) algorithm.

\begin{algorithm}[h]
	\caption{SVAG}
	\label{alg:svag}
	\begin{algorithmic}
		\STATE {\bfseries input} single valued operators $R_i: \R^N\to\R^N$, initial state $x^0\in\R^N$ and $y_1^0,\dots,y_n^0\in\R^N$, step-size $\lambda > 0$, innovation weight $\theta \in \R$
		\FOR{$k = 0,1,\dots$}
		\STATE Sample $i^k$ uniformly from $\{1,\dots ,n\}$
		\STATE $x^{k+1} = x^k - \lambda \left(\frac{\theta}{n}\big(R_{i^k}x^k - y_{i^k}^k \big)  + \frac{1}{n}\sum_{j=1}^n y_j^k \right)$
		\STATE $y_{i^k}^{k+1} = R_{i^k}x^k$
		\STATE $y_j^{k+1} = y_j^k$ {for all} $j \neq i^k$
		\ENDFOR
	\end{algorithmic}
\end{algorithm}

SVAG is heavily inspired by SAG and SAGA with both being special cases, $\theta = 1$ and $\theta = n$ respectively.
Just like SAG and SAGA, in each iteration, SVAG evaluates one operator $R_{i^k}$ and stores the results in $y_{i^k}^{k+1}$.
An estimate of the full operator is then formed as
\begin{align*}
	Rx^k \approx \widetilde{R}^k = \tfrac{\theta}{n} (R_{i^k}x^k - y_{i^k}^k) + \tfrac{1}{n}\sum_{j=1}^n y_j^k.
\end{align*}
The scalar $\theta$ determine how much weight should be put on the new information gained from evaluating $R_{i^k}x^k$.
If the innovation, $R_{i^k}x^k - y_{i^k}^k$, is highly correlated with the total innovation, $Rx^k - \tfrac{1}{n}\sum_{j=1}^n y_j^k$, a large innovation weight $\theta$ can be chosen and vice versa.
The innovation weight $\theta$ also determines the bias of SVAG.
Taking the expected value $\widetilde{R}^k$ given the information at iteration $k$ gives
\begin{align*}
	\Exp[\widetilde{R}^k|x^k, y_1^k,\dots,y_n^k] = \tfrac{\theta}{n} Rx^k + (1-\tfrac{\theta}{n})\tfrac{1}{n} \sum_{j=1}^n y_j^k
\end{align*}
which reveals that $\widetilde{R}^k$ is an unbiased estimate of $Rx^k$ if $\theta=n$, i.e., in the SAGA case.
Any other choice, for instance SAG where $\theta = 1$, yields a bias towards $\tfrac{1}{n}\sum_{j=1}^n y_j^k$.

\subsection{Contribution}
The theory behind finding roots of monotone operators in general, and cocoercive operators in particular, has been put to good use when analyzing first order optimization methods, examples include \cite{rockafellarMonotoneOperatorsProximal1976,bauschke_convex_2017,latafatPrimalDualProximalAlgorithms2018,briceno-ariasForwardBackwardHalfForwardAlgorithm2018,davisThreeOperatorSplittingScheme2017,tsengModifiedForwardBackwardSplitting2000}.
For instance can both the proximal-gradient and ADMM methods be seen as instances of classic root-finding fixed-point iterations and analyzed as such, namely forward-backward and Douglas--Rachford respectively.
The resulting analyses can often be simple and intuitive and even though the root-finding formulation is more general---not all cocoercive operators are gradients of convex functions---they are not necessarily more conservative.
For example, analyzing proximal-gradient as forward-backward splitting yield the same rates and step-size conditions as analyzing it as a minimization method in the smooth/cocoercive setting, see for instance \cite[Theorem 2.1.14]{nesterov_introductory_2004} and \cite[Example 5.18 and Proposition 4.39]{bauschke_convex_2017}.
However, the main contribution of this paper is to show that the same is not true for VR-SG methods, in particular it is not true for SVAG when it is biased.

The results consist of two main convergence theorems for SVAG:
one in the cocoercive operator case and one in the cocoercive gradient case, the later being equivalent to the minimization of a smooth and convex finite sum.
Both of these theorems match or improve upon previously known results for the SAG and SAGA special cases.
Comparing the two settings reveal that SVAG can use significantly larger step-sizes, with faster convergence as a result, in the cocoercive gradient case compared to the general cocoercive operator case.
In the operator case, an upper bound on the step-size that scales as $\mathcal{O}(n^{-1})$ is found where $n$ is the number of terms in \eqref{eq:zerprob}.
However, the restrictions on the step-size loosen with reduced bias and the unfavorable $\mathcal{O}(n^{-1})$ scaling disappears completely when SVAG is unbiased.
In the gradient case, this bad scaling never occurs, regardless of bias.
We provide examples in which SVAG diverges with step-sizes larger than the theoretical upper bounds in the operator case.
Since the gradient case is proven to converge with much larger step-sizes, this verifies the difference between the convergence behavior of cocoercive operators and gradients.

These results indicate that it is inadvisable to only rely on more general monotone operator theory and not explicitly use the gradient property when analyzing VR-SG methods meant for optimization.
However, the large impact of bias in the cocoercive operator setting also raises the question regarding its importance in other non-gradient settings as well.
One such setting of interest, where the operators are not gradients of convex functions, is the case of saddle-point problems.
These problems are of importance in optimization due to their use in primal-dual methods but recently they have also gained a lot of attention due to their applications in the training of GANs in machine learning.
Because of this, and due to the attractive properties of VR-SG methods in the convex optimization setting, efforts have gone into applying VR-SG methods to saddle-point problems as well \cite{palaniappanStochasticVarianceReduction2016,shiBregmanDivergenceStochastic2017,chavdarovaReducingNoiseGAN2019,carmonVarianceReductionMatrix2019,zhangUnifyingFrameworkVariance2019}.
Most of these efforts have been unbiased, something our analysis suggests is wise.
With that said, it is important to note that our analysis is often not directly applicable due the fact that saddle-point problems rarely are cocoercive.

The main reason for the recent rise in popularity of variance-reduced stochastic methods is their use in the optimization setting, but, although bias plays a big role in the cocoercive operator case, our results are not as clear in this setting.
For instance, the theoretical results for the SAG and SAGA special cases yield identical rates and step-size conditions with no clear advantage to either special case.
Further experiments are therefore performed where several different choices of bias in SVAG are examined on a set of logistic regression and SVM optimization problems.
However, the results of these experiments are in line with existing works with no significant advantage of any particular bias choice in SVAG, these choices include both SAG and SAGA.
Although the performance difference is significant in some cases, no single choice of bias performs best for all problems and all bias choices eventually converge with the same rate in the majority of the cases.
Furthermore, the theoretical maximal step-size can routinely be exceeded in these experiments, indicating that there is room for further theoretical improvements.

\subsection{Related Work}
There is a large array of options for solving \eqref{eq:zerprob}.
For $n \in \{1,2,3,4\}$, several operator splitting methods exist with varying assumptions on the operator properties, see for instance \cite{goldsteinConvexProgrammingHilbert1964,levitinConstrainedMinimizationMethods1966,lionsSplittingAlgorithmsSum1979,giselssonNonlinearForwardBackwardSplitting2021,tsengModifiedForwardBackwardSplitting2000,briceno-ariasForwardBackwardHalfForwardAlgorithm2018} and the references therein.
However, while these methods also can be applied for larger $n$ by simply regrouping the terms, they do not utilize the finite sum structure of the problem.
Algorithms have therefore been designed to utilize this structure for arbitrary large $n$ with the hopes of reducing the total computational costs, e.g., \cite{combettesPrimalDualSplittingAlgorithm2012,raguetGeneralizedForwardBackwardSplitting2013,combettesAsynchronousBlockiterativePrimaldual2018,combettesSolvingCompositeFixed2020}.
In particular the problem and method in \cite{combettesSolvingCompositeFixed2020} is closely related to the root-finding problem and algorithm considered in this paper.

Using the notation of \cite{combettesSolvingCompositeFixed2020}, when $T_0 = \Id$, the fixed point problem of \cite{combettesSolvingCompositeFixed2020} can be mapped to \eqref{eq:zerprob} via $R_i = \omega_i(\Id - T_i)$ and vice verse.
\footnote{If $T_i$ is $\alpha_i$-averaged, as assumed in \cite{combettesSolvingCompositeFixed2020}, $R_i$ is $(2\alpha_i\omega_i)^{-1}$-cocoercive.}
Many applications considered in \cite{combettesSolvingCompositeFixed2020} can therefore, at least in part, be tackled with our algorithm as well.
In particular, the problem of finding common fixed points of firmly nonexpansive operators can directly be solved by our algorithm.
However, \cite{combettesSolvingCompositeFixed2020} is more general in that it allows for $T_0 \neq \Id$ and works in general real Hilbert spaces.
Comparing with the algorithm of \cite{combettesSolvingCompositeFixed2020} we see that, just as our algorithm is a generalization of SAG/SAGA, it can be seen as a generalization of Finito \cite{defazioFinitoFasterPermutable2014}, another classic VR-SG method.
It generalize Finito in several way, for instance it allows for an additional proximal/backward step and it replace the stochastic selection with a different selection criteria.
However, in the optimization setting it still suffers from the same drawback as Finito when compared to SAG/SAGA-like algorithms.
It still needs to store a full copy of the iterate for each term in objective.
Since SAG, SAGA, and SVAG only need to store the gradient of each term, they can utilize any potential structure of the gradients to reduce the storage requirements \cite{le_roux_stochastic_2012}.
Although the differences above are interesting in their own right, the notion of bias we examine in this paper is not applicable to Finito-like algorithms.

SAG and SAGA were compared in \cite{defazio_saga_2014} but with no direct focus on the effects of bias.
Other examples of research on SAG and SAGA include acceleration, sampling strategy selection, and ways to reduce the memory requirement \cite{hofmann_variance_2015,schmidt_non-uniform_2015,gowerStochasticQuasigradientMethods2021,qian_saga_2019,morinSamplingUpdateFrequencies2020,zhouDirectAccelerationSAGA2019}.
However, none of these works, including \cite{morinSamplingUpdateFrequencies2020} that was written by the authors, analyze the biased case we consider in this paper.
Even the works considering non-uniform sampling of gradients \cite{schmidt_non-uniform_2015,gowerStochasticQuasigradientMethods2021,qian_saga_2019,morinSamplingUpdateFrequencies2020} perform some sort of bias correction in order to remain unbiased.
Furthermore, in order to keep the focus on the effects of the bias we have refrained from bringing in such generalizations into this work, making it distinct from the above research.
To the authors' knowledge, the only theoretical convergence result for biased VR-SG methods are the ones for SAG \cite{le_roux_stochastic_2012,schmidt_minimizing_2017}.
But, since they only consider SAG, they fail to capture the breadth of SVAG and our proof is the first to simultaneously capture SAG, SAGA, and more.

Since the release of the first preprint of this paper, \cite{driggsBiasedStochasticGradient2020} has also provided a proof covering the gradient case of both SAG and SAGA, and some choices of bias in SVAG.
All though \cite{driggsBiasedStochasticGradient2020} does not consider cocoercive operators, it is some sense more general with them considering a general biased stochastic estimator of the gradient.
This generality comes at the cost of a more conservative analysis with their step-size scaling with $\mathcal{O}(n^{-1})$ in all cases.

\section{Preliminaries and Notation}\label{sec:prel}
Let $\R$ denote the real numbers and let the natural numbers be denoted $\N = \{0,1,2,\dots\}$.
Let $\inprod{\cdot}{\cdot}$ denote the standard Euclidean inner product and $\norm{\cdot} = \sqrt{\inprod{\cdot}{\cdot}}$ the standard 2-norm.
The scaled inner product and norm on we denote as $\inprod{\cdot}{\cdot}_\Sigma = \inprod{\Sigma(\cdot)}{\cdot}$ and $\norm{\cdot}_{\Sigma} = \sqrt{\inprod{\cdot}{\cdot}_{\Sigma}}$ where $\Sigma$ is a positive definite matrix.
If $\Sigma$ is not positive definite, $\norm{\cdot}_{\Sigma}$ is not a norm but we keep the notation for convenience.

Let $n$ be the number of operators in \eqref{eq:zerprob}.
The vector $\mathbf{1}$ is the vector of all ones in $\R^n$ and $e_i$ be the vector in $\R^n$ of all zeros except the $i$:th element that contains a $1$.
The matrix $I$ is an identity matrix with the size derived from context and $E_i = e_ie_i^T$.

The symbol $\otimes$ denotes the Kronecker product of two matrices.
The Kronecker product is linear in both arguments and the following properties hold
\begin{equation*}
	(A\otimes B)^T = A^T \otimes B^T, \qquad (A\otimes B)(C\otimes D) = (AC)\otimes (BD).
\end{equation*}
In the last property it is assumed that the dimensions are such that the matrix multiplications are well defined.
The eigenvalues of $A\otimes B$ are given by
\begin{align}\label{eq:kroneig}
	\tau_i\mu_j \text{ for all } i \in \{1,\dots ,m\}, j \in \{ 1,\dots ,l \}
\end{align}
where $\tau_i$ and $\mu_j$ are the eigenvalues of $A$ and $B$ respectively.

The Cartesian product of two sets $C_1$ and $C_2$ is defined as
\begin{align*}
	C_1\times C_2 = \{(c_1,c_2) \mid c_1\in C_1, c_2\in C_2\}.
\end{align*}
From this definition we see that if $C_1$ and $C_2$ are closed and convex, so is $C_1\times C_2$.

Let $X^\star$ be the set of all solutions of $\eqref{eq:zerprob}$,
\begin{align*}
	X^\star =  \{x\,|\, 0 = {\textstyle\frac{1}{n}\sum_{i=1}^n} R_ix \}
\end{align*}
and define $Z^\star$ as the set of primal-dual solutions
\begin{align*}
	Z^\star = \{(x, R_1x,\dots , R_nx) \,|\, 0 = {\textstyle\frac{1}{n}\sum_{i=1}^n} R_ix \}.
\end{align*}
Assuming they exists, $x^\star$ denotes a solution to \eqref{eq:zerprob} and $z^\star$ denotes a primal-dual solution, i.e., $x^\star \in X^\star$ and $z^\star \in Z^\star$.

A single valued operator $R: \R^N \to \R^N$ is $\frac{1}{L}$-cocoercive if
\begin{equation}\label{eq:cocoercivity}
	\begin{aligned}
		&\inprod{Rx - Ry}{x-y}
		\geq \tfrac{1}{L} \norm{Rx - Ry}^2
	\end{aligned}
\end{equation}
holds for all $x,y \in \R^N$.
An operator that is $\frac{1}{L}$-cocoercive is $L$-Lipschitz continuous.
The set of zeros of a cocoercive operator $R$ is closed and convex.

A differentiable convex function $f: \R^N \to \R$ is called $L$-smooth if the gradient is $\frac{1}{L}$-cocoercive.
Equivalently, a differentiable convex function is $L$-smooth if
\begin{equation}\label{eq:descent}
	f(y) \leq f(x) + \inprod{\nabla f(x)}{y-x} + \tfrac{L}{2}\norm{y-x}^2
\end{equation}
holds for all $x,y \in \R^N$.

If $f_i: \R^N \to \R$ is a differentiable convex functions for each $i\in\{1,\dots,n\}$, the minimization of $\sum_{i=1}^n f_i(x)$ is equivalent to \eqref{eq:zerprob} with $R_i = \nabla f_i$.

For more details regarding monotone operators and convex functions see \cite{nesterov_introductory_2004,bauschke_convex_2017}.

To establish almost sure sequence convergence of the stochastic algorithm, the following propositions will be used.
The first is from \cite{robbins_convergence_1971} and establishes convergence of non-negative almost super-martingales.
The second is based on \cite{combettes_stochastic_2015} and provides the tool to show almost sure sequence convergence.
\begin{prop}\label{prop:supermartingale}
	Let $(\Omega, \mathcal{F}, P)$ be a probability space and $\mathcal{F}_0 \subset \mathcal{F}_1 \subset \dots$ be a sequence of sub-$\sigma$-algebras of $\mathcal{F}$.
	For all $k \in \N$, let $z^k$, $\beta^k$, $\xi^k$ and $\zeta^k$ be non-negative $\mathcal{F}_k$-measurable random variables.
	If $\sum_{i=0}^\infty \beta^i < \infty$, $\sum_{i=0}^\infty \xi^i < \infty$ and
	\begin{align*}
		\Exp[z^{k+1}|\mathcal{F}_k] \leq (1 + \beta^k)z^k + \xi^k - \zeta^k
	\end{align*}
	hold almost surely for all $k \in \N$, then $z^k$ converges a.s.\ to a finite valued random variable and $\sum_{i=0}^\infty \zeta^i < \infty$ almost surely.
	\begin{proof}
		See \cite[Theorem 1]{robbins_convergence_1971}.
	\end{proof}
\end{prop}
\begin{prop}\label{prop:stochseqconv}
	Let $Z$ be a non-empty closed subset of a finite dimensional Hilbert space $H$, let $\phi: [0,\infty) \to [0,\infty)$ be a strictly increasing function such that $\phi(t) \to \infty$ as $t \to \infty$, and let $(x^k)_{k\in \N}$ be a sequence of $H$-valued random variables.
	If $\phi(\norm{x^k - z})$ converges a.s.\ to a finite valued non-negative random variable for all $z \in Z$, then the following hold:
	\begin{enumerate}[(i)]
		\item\ $(x^k)_{k\in\N}$ is bounded almost surely.
		\item Suppose the cluster points of $(x^k)_{k\in\N}$ are a.s.\ in $Z$, then $(x^k)_{k\in\N}$ converge a.s.\ to a $Z$-valued random variable.
	\end{enumerate}
	\begin{proof}
		In finite dimensional Hilbert spaces, these two statements are the same as statements (ii) and (iv) of \cite[Proposition 2.3]{combettes_stochastic_2015}.
		Hence, consider the proof of \cite[Proposition 2.3]{combettes_stochastic_2015} restricted to finite dimensional Hilbert spaces.
		The proof of (ii) in \cite[Proposition 2.3]{combettes_stochastic_2015} only relies on the a.s.\ convergence of  $\phi(\norm{x^k - z})$ and hence is implied by the assumptions of this proposition.
		This proves our first statement.
		The proof of (iv) in \cite[Proposition 2.3]{combettes_stochastic_2015} only relies on (iii) of \cite[Proposition 2.3]{combettes_stochastic_2015} which in turn is implied by (ii) of \cite[Proposition 2.3]{combettes_stochastic_2015}, i.e., our first statement.
		This proves our second statement.
	\end{proof}
\end{prop}

\section{Convergence}\label{sec:convergence}
Throughout the analysis we will use the following two assumptions on the operators of \eqref{eq:zerprob}.
\begin{ass}\label{ass:op}
	For each $i\in\{1,\dots,n\}$, let $R_i$ be $\frac{1}{L}$-cocoercive and $X^\star \neq \emptyset$, i.e., \eqref{eq:zerprob} has at least one solution.
\end{ass}
\begin{ass}\label{ass:grad}
	For each $i\in\{1,\dots,n\}$, let $R_i = \nabla f_i$ for some differentiable function $f_i$ and define $F = \tfrac{1}{n}\sum_{i=1}^n f_i$.
	Furthermore, let \assref{ass:op} hold, i.e., $f_i$ is $L$-smooth and convex and $\argmin F(x)$ exists.
\end{ass}

\subsection{Reformulation}
We begin by formalizing and reformulating \algref{alg:svag} into a more convenient form.
Let $(\Omega, \mathcal{F}, P)$ be the underlying probability space of \algref{alg:svag}.
The index selected at iteration $k$ is then a uniformly distributed random variable $i^k: \Omega \to \{1,\dots,n\}$.
For each $k\in\N$, define the random variable $z^k:\Omega\to\R^{N(n+1)}$ as $z^k = (x^k,y_1^k,\dots ,y_n^k)$ where $x^k$ and $y_i^k$ for $i \in \{1,\dots,n\}$ are the iterates of \algref{alg:svag}.
Let $\mathcal{F}_0 \subset \mathcal{F}_1 \subset \dots$ be a sequence of sub-$\sigma$-algebras of $\mathcal{F}$ such that $z^k$ are $\mathcal{F}_k$-measurable and $i^k$ is independent of $\mathcal{F}_k$.
With the operator $\mathbf{B} : \R^{N(n+1)} \to \R^{2Nn}$ defined as $\mathbf{B}(x,y_1,\dots,y_n) = (R_1x,\dots ,R_nx,y_1,\dots ,y_n)$, one iteration of \algref{alg:svag} can be written as
\begin{align}\label{eq:zupdate}
	z^{k+1} = z^k - (U_{i^k}\otimes I)\mathbf{B}z^k
\end{align}
where $z^0 \in \R^{N(n+1)}$ is given and
\begin{align*}
	U_i =
	\begin{bigmatrix}
		\frac{\lambda}{n}\theta e_i^T & -\frac{\lambda}{n}\theta e_i^T+\frac{\lambda}{n}\mathbf{1}^T \\
		-E_i & E_i
	\end{bigmatrix}
\end{align*}
for all $i\in\{1,\dots,n\}$.
The vector $e_i$ and the matrix $E_i$ are defined in \secref{sec:prel}.

The following lemma characterizes the zeros of $(U_i\otimes I)\mathbf{B}$ and hence the fixed points of \eqref{eq:zupdate} and \algref{alg:svag}.
\begin{lem}\label{lem:zeros}
	Let \assref{ass:op} hold, each $z^\star$ in $Z^\star$ is then a zero of $(U_i\otimes I)\mathbf{B}$ for all $i\in\{1,\dots,n\}$, i.e.
	\begin{align*}
		\forall z^\star \in Z^\star, \forall i \in \{1,\dots ,n\}: \quad 0 = (U_i\otimes I)\mathbf{B}z^\star.
	\end{align*}
	Furthermore, the set $Z^\star$ is closed and convex and $R_ix^\star = R_i\bar{x}^\star$ for all $x^\star,\bar{x}^\star \in X^\star$ and for all $i\in\{1,\dots,n\}$.
\end{lem}
\begin{proof}[Proof of \lemref{lem:zeros}]
	The zero statement, $0 = (U_i\otimes I)\mathbf{B}z^\star$, follows from definition of $z^\star$.
	For closedness and convexity of $Z^\star$, we first prove that $R_ix^\star$ is unique for each $i\in\{1,\dots,n\}$.
	Taking $x,y \in X^\star$, which implies ${\textstyle\sum_{i=1}^n} R_ix = {\textstyle\sum_{i=1}^n}R_iy = 0$, and using cocoercivity \eqref{eq:cocoercivity} of each $R_i$ gives
	\begin{gather*}
		0
		= \inprod{{\textstyle\sum_{i=1}^n} R_ix - {\textstyle\sum_{i=1}^n} R_iy}{x-y}
		= {\textstyle\sum_{i=1}^n} \inprod{R_ix - R_iy}{x-y} \\
		\geq {\textstyle\sum_{i=1}^n} \tfrac{1}{L}\norm{R_ix - R_iy}^2 \geq 0,
	\end{gather*}
	hence must $R_ix = R_iy$ for all $i\in\{1,\dots,n\}$.
	The set $Z^\star$ is a Cartesian product of $X^\star$ and the points $r_i = R_ix^\star$ for $i\in\{1,\dots,n\}$ and any $x^\star\in X^\star$.
	A set consisting of only one point is closed and convex and $X^\star$ is closed and convex since $\frac{1}{n}\sum_{i=1}^nR_i$ is cocoercive \cite[Proposition 23.39]{bauschke_convex_2017}, hence is $Z^\star$ closed and convex.
\end{proof}

The operator $\mathbf{B}$ in the reformulated algorithm can be used to enforce the following property on the sequence $(z^k)_{k\in\N}$.
\begin{lem}\label{lem:clusterpoints}
	Let $(\Omega, \mathcal{F}, P)$ be a probability space and $(z^k)_{k\in\N}$ be a sequence of random variables $z^k: \Omega \to \R^{N(n+1)}$.
	If $\mathbf{B}z^k \to \mathbf{B}z^\star$ a.s.\ where $z^\star \in Z^\star$, then any cluster point of $(z^k)_{k\in\N}$ will almost surely be in $Z^\star$.
\end{lem}
\begin{proof}[Proof of \lemref{lem:clusterpoints}]
	Let $z$ be a cluster point of $(z^k)_{k\in\N}$.
	Take an $\omega \in \Omega$ such that $\mathbf{B}z^k(\omega) \to \mathbf{B}z^\star$.
	For this $\omega$ and for all $k\in\N$, we define the realizations of $z$ and $z^k$ as
	\begin{align*}
		z(\omega) = (\bar{x},\bar{y}_1,\dots,\bar{y}_n)
		,\quad
		z^k(\omega) = (\bar{x}^k,\bar{y}_1^k,\dots,\bar{y}_n^k)
	\end{align*}
	where $\bar{x},\bar{y}_1,\dots,\bar{y}_n\in\R^N$ and $\bar{x}^k,\bar{y}_1^k,\dots,\bar{y}_n^k\in\R^N$ for all $k\in\N$.

	Since $\mathbf{B}\bar{z}^k \to \mathbf{B}z^\star$ we directly have $\bar{y}_i^k \to R_i x^\star$ for $x^\star\in X^\star$ and hence must $\bar{y}_i = R_ix^\star$ for all $i\in\{1,\dots,n\}$.
	Note, $R_ix^\star$ is independent of which $x^\star\in X^\star$ was chosen, see \lemref{lem:zeros}.
	Furthermore, $\mathbf{B}\bar{z}^k \to \mathbf{B}\bar{z}^\star$ implies that $R_i\bar{x}^k \to R_i x^\star$ for all $i\in\{1,\dots,n\}$.
	Let $(\bar{x}^{k(l)})_{l\in\N}$ be a subsequence converging to $\bar{x}$, then
	\begin{gather*}
		\norm{{\textstyle\frac{1}{n}\sum_{i=1}^n R_i}\bar{x}}
		\leq \norm{{\textstyle\frac{1}{n}\sum_{i=1}^nR_i}\bar{x}^{k(l)} - {\textstyle\frac{1}{n}\sum_{i=1}^nR_i}\bar{x}} + \norm{{\textstyle\frac{1}{n}\sum_{i=1}^nR_i}\bar{x}^{k(l)}} \\
		\leq L\norm{\bar{x}^{k(l)} - \bar{x}} + \norm{{\textstyle\frac{1}{n}\sum_{i=1}^nR_i}\bar{x}^{k(l)}} \to \norm{{\textstyle\frac{1}{n}\sum_{i=1}^nR_i}x^\star} = 0
	\end{gather*}
	as $l \to \infty$ where $L$-Lipschitz continuity of $\frac{1}{n}\sum_{i=1}^nR_i$ was used.
	This concludes that $\bar{x} \in X^\star$ and since $\bar{y}_i = R_ix^\star = R_i\bar{x}$ for all $i\in\{1,\dots,n\}$ by \lemref{lem:zeros}, we have that $z(\omega) \in Z^\star$.
	Since this hold for any $\omega$ such that $\mathbf{B}z^k(\omega) \to \mathbf{B}z^\star$ and the set in $\mathcal{F}$ of all such $\omega$ have probability one due to the almost sure convergence of $\mathbf{B}z^k \to \mathbf{B}z^\star$, we have $z\in Z^\star$ almost surely.
\end{proof}

The reformulation \eqref{eq:zupdate} further allows us to concisely formulate two Lyapunov inequalities.
\begin{lem}\label{lem:fejer-cond}
	Let \assref{ass:op} hold, the update \eqref{eq:zupdate} then satisfies
	\begin{align*}
		&\Exp[\norm{z^{k+1} - z^\star}_{H\otimes I}^2 | \mathcal{F}_k]\\
		&\quad\leq \norm{z^k - z^\star}_{H\otimes I}^2 - \norm{\mathbf{B}z^k - \mathbf{B}z^\star}_{(2 M  - \Exp[U_{i^k}^THU_{i^k}] - \xi I) \otimes I }^2 \\
		&\qquad- {\xi n L}\inprod{Rx^k}{x^k - x^\star}
	\end{align*}
	for all $k\in\N$ and $\xi \in [0,\frac{2\lambda}{nL}]$,
	where matrices $H$ and $M$ are given by
	\begin{align*}
		H = \begin{bigmatrix}
			1 & - \frac{\lambda}{n}(n-\theta) \mathbf{1}^T \\
			-\frac{\lambda}{n}(n-\theta)\mathbf{1} & \frac{\lambda}{L} I + \frac{\lambda^2}{n^2}(n-\theta)^2 \mathbf{11}^T
		\end{bigmatrix}
	\end{align*}
	and
	\begin{align*}
		M
		=
		\begin{bigmatrix}
			2 & -1 \\
			-1 & 2
		\end{bigmatrix} \otimes \frac{1}{2n} \frac{\lambda}{L} I
		-
		\begin{bigmatrix}
			0 & 1\\
			1 & 0
		\end{bigmatrix} \otimes \frac{\lambda^2}{2 n^2}(n-\theta)\mathbf{11}^T.
	\end{align*}
\end{lem}
\begin{lem}\label{lem:funcdecrease}
	Let \assref{ass:grad} hold, the update \eqref{eq:zupdate} then satisfies
	\begin{align*}
		\Exp[F((K\otimes I) z^{k+1})|\mathcal{F}_k] \leq F((K\otimes I) z^k) - \norm{\mathbf{B}z^k - \mathbf{B}z^\star}^2_{\frac{1}{2}S\otimes I}
	\end{align*}
	for all $k\in\N$, where $K = \begin{bigmatrix} 1 & \frac{\lambda}{n}\mathbf{1}^T \end{bigmatrix}$ and
	\begin{align*}
		S
		&=
		\begin{bigmatrix}
			2 & -1 \\
			-1 & 0
		\end{bigmatrix} \otimes (\theta - 1)\frac{\lambda}{n^3} \mathbf{11}^T
		-
		\begin{bigmatrix}
			1  & -1  \\
			-1  & 1
		\end{bigmatrix} \otimes (\theta - 1)^2\frac{L\lambda^2}{n^3}  I
		\\
		&\quad+
		\begin{bigmatrix}
			0 & 1 \\
			1 & 0
		\end{bigmatrix} \otimes \frac{\lambda}{n^2} \mathbf{11}^T.
	\end{align*}
\end{lem}

\begin{proof}[Proof of \lemref{lem:fejer-cond}]
	Take $k\in\N$, note that since $U_{i^k}$ is independent of $\mathcal{F}_k$ and $z_k$ is $\mathcal{F}_k$-measurable we have
	\begin{align*}
		&\Exp[\inprod{(U_{i^k}\otimes I)  (\mathbf{B}z^k - \mathbf{B}z^\star)}{z^k-z^\star}_{H\otimes I} |\mathcal{F}_k] \\
		&\quad = \inprod{(H\Exp [U_{i^k}]\otimes I)  (\mathbf{B}z^k - \mathbf{B}z^\star)}{z^k-z^\star}.
	\end{align*}
	The matrix $H\Exp [U_{i^k}]$ is given by
	\begin{align*}
		H\Exp [U_{i^k}]
		=
		\begin{bigmatrix}
			\frac{\lambda}{n}\mathbf{1}^T & 0 \\
			- \frac{\lambda^2}{n^2}(n-\theta)\mathbf{11}^T - \frac{\lambda}{nL}I & \frac{\lambda}{nL}I
		\end{bigmatrix},
	\end{align*}
	see the supplementary material for verification of this and other matrix identities.
	We also note that
	\begin{align*}
		\inprod{Rx^k - Rx^\star}{x^k -x^\star}
		=
		\inprod{(\begin{bigmatrix} \tfrac{1}{n}\mathbf{1}^T & 0 \\ 0 & 0 \end{bigmatrix} \otimes I) (\mathbf{B}z^k - \mathbf{B}z^\star)}{z^k-z^\star}.
	\end{align*}
	Taking $\xi \in [0,\frac{2\lambda}{nL}]$ and putting these two expression together yield
	\begin{align*}
		&\Exp[\inprod{(U_{i^k}\otimes I)  (\mathbf{B}z^k - \mathbf{B}z^\star)}{z^k-z^\star}_{H\otimes I} |\mathcal{F}_k] - \tfrac{\xi n L}{2}\inprod{Rx^k - Rx^\star}{x^k -x^\star}\\
		&\quad=\inprod{(
			\begin{bigmatrix}
				(\frac{\lambda}{n} -\frac{\xi L}{2})\mathbf{1}^T & 0 \\
				- \frac{\lambda^2}{n^2}(n-\theta)\mathbf{11}^T - \frac{\lambda}{nL}I & \frac{\lambda}{nL}I
			\end{bigmatrix}
		\otimes I) (\mathbf{B}z^k - \mathbf{B}z^\star)}{z^k-z^\star}.
	\end{align*}
	Using $\frac{1}{L}$-cocoercivity of $R_i$ for each $i\in\{1,\dots,n\}$ gives
	\begin{align*}
		&\Exp[\inprod{(U_{i^k}\otimes I)  (\mathbf{B}z^k - \mathbf{B}z^\star)}{z^k-z^\star}_{H\otimes I} |\mathcal{F}_k] - \tfrac{\xi n L}{2}\inprod{Rx^k - Rx^\star}{x^k -x^\star}\\
		&\quad\geq \inprod{(\begin{bigmatrix} (\frac{\lambda}{nL} -\frac{\xi}{2}) I & 0 \\ - \frac{\lambda^2}{n^2}(n-\theta)\mathbf{11}^T - \frac{\lambda}{nL}I & \frac{\lambda}{nL}I \end{bigmatrix} \otimes I) (\mathbf{B}z^k - \mathbf{B}z^\star)}{\mathbf{B}z^k-\mathbf{B}z^\star}
	\end{align*}
	Setting
	\begin{align*}
		\bar{M} =
		\begin{bigmatrix}
			\frac{\lambda}{nL} I & 0 \\
			- \frac{\lambda^2}{n^2}(n-\theta)\mathbf{11}^T - \frac{\lambda}{nL}I & \frac{\lambda}{nL}I
		\end{bigmatrix}
	\end{align*}
	gives
	\begin{align*}
		&\Exp[\inprod{(U_{i^k}\otimes I)  (\mathbf{B}z^k - \mathbf{B}z^\star)}{z^k-z^\star}_{H\otimes I} |\mathcal{F}_k] - \tfrac{\xi n L}{2}\inprod{Rx^k - Rx^\star}{x^k -x^\star}\\
		&\quad\geq \inprod{(\bar{M}\otimes I)(\mathbf{B}z^k - \mathbf{B}z^\star)}{\mathbf{B}z^k-\mathbf{B}z^\star} \\
		&\qquad- \inprod{(\begin{bigmatrix} \frac{\xi}{2} I & 0 \\ 0 & 0 \end{bigmatrix} \otimes I)(\mathbf{B}z^k - \mathbf{B}z^\star)}{\mathbf{B}z^k-\mathbf{B}z^\star}\\
		&\quad\geq \norm{\mathbf{B}z^k - \mathbf{B}z^\star}^2_{\frac{1}{2}(\bar{M} + \bar{M}^T) \otimes I} - \tfrac{\xi}{2}\norm{\mathbf{B}z^k-\mathbf{B}z^\star}^2 \\
		&\quad= \norm{\mathbf{B}z^k - \mathbf{B}z^\star}^2_{(M - \frac{\xi}{2}I)\otimes I}
	\end{align*}
	where $M = \frac{1}{2}(\bar{M} + \bar{M}^T)$ is the matrix in the Lemma.
	Finally, using this inequality and $0 = (U_{i^k}\otimes I)\mathbf{B}z^\star$ from \lemref{lem:zeros} gives
	\begin{align*}
		&\Exp[\norm{z^{k+1} - z^\star}_{H\otimes I}^2|\mathcal{F}_k] \\
		&\quad= \Exp[\norm{\big(z^k - (U_{i^k}\otimes I)\mathbf{B}z^k\big) - \big(z^\star - (U_{i^k}\otimes I)\mathbf{B}z^\star\big) }_{H\otimes I}^2 |\mathcal{F}_k]\\
		&\quad= \norm{z^k - z^\star}_{H\otimes I}^2 + \Exp[\norm{(U_{i^k}\otimes I)(\mathbf{B}z^k - \mathbf{B}z^\star)}_{H\otimes I}^2|\mathcal{F}_k] \\
		&\quad\quad- 2\Exp[\inprod{(U_{i^k}\otimes I)(\mathbf{B}z^k - \mathbf{B}z^\star)}{z^k-z^\star}_{H\otimes I} |\mathcal{F}_k]\\
		&\quad\leq \norm{z^k - z^\star}_{H\otimes I}^2 + \norm{\mathbf{B}z^k - \mathbf{B}z^\star}_{\Exp[U_{i^k}^THU_{i^k} ] \otimes I}^2 \\
		&\quad\quad - \norm{\mathbf{B}z^k - \mathbf{B}z^\star}^2_{(2M - \xi I)\otimes I} - {\xi n L}\inprod{Rx^k - Rx^\star}{x^k - x^\star} \\
		&\quad= \norm{z^k - z^\star}_{H\otimes I}^2 - \norm{\mathbf{B}z^k - \mathbf{B}z^\star}_{(2 M  - \Exp[U_{i^k}^THU_{i^k}] - \xi I) \otimes I }^2 \\
		&\qquad- {\xi n L}\inprod{Rx^k}{x^k - x^\star}.
	\end{align*}
\end{proof}

\begin{proof}[Proof of \lemref{lem:funcdecrease}]
	Take $k\in\N$ and note that
	\begin{align*}
		(K\otimes I)z^{k+1} =(K\otimes I)(z^k - (U_{i^k}\otimes I)\mathbf{B}z^k) = x^k - (Q_{i^k} \otimes I)\mathbf{B}z^k
	\end{align*}
	where $Q_{i^k} = \frac{\lambda}{n} \begin{bigmatrix} (\theta - 1)e_{i^k}^T & -(\theta - 1)e_{i^k}^T \end{bigmatrix}$.
	Furthermore, with $G = \tfrac{1}{n}\begin{bmatrix} \mathbf{1}^T & 0 \end{bmatrix}$, we have $\nabla F(x^k) = (G \otimes I)\mathbf{B}z^k$.
	From the definition of $z^\star$ we have $0 =(G\otimes I)\mathbf{B}z^\star = (Q_{i^k} \otimes I)\mathbf{B}z^\star$.
	Using $L$-smoothness, \eqref{eq:descent}, of $F$ yields
	\begin{align*}
		&\Exp[F((K\otimes I)z^{k+1})|\mathcal{F}_k] \\
		&\quad= \Exp[F(x^k - (Q_{i^k} \otimes I)\mathbf{B}z^k)|\mathcal{F}_k]\\
		&\quad\leq F(x^k) - \inprod{\nabla F(x^k)}{(\Exp[Q_{i^k}] \otimes I)\mathbf{B}z^k} + {\textstyle\frac{L}{2}}\Exp[\norm{(Q_{i^k} \otimes I)\mathbf{B}z^k}^2|\mathcal{F}_k] \\
		&\quad= F(x^k) - \inprod{(G\otimes I)\mathbf{B}z^k}{(\Exp[Q_{i^k}] \otimes I)\mathbf{B}z^k} + \norm{\mathbf{B}z^k}_{\frac{L}{2}\Exp[Q_{i^k}^TQ_{i^k}] \otimes I}^2 \\
		&\quad= F(x^k) - \norm{\mathbf{B}z^k}_{\frac{1}{2}\Exp[Q_{i^k}^T G + G^T Q_{i^k}]\otimes I}^2 + \norm{\mathbf{B}z^k}_{\frac{L}{2}\Exp[Q_{i^k}^TQ_{i^k}] \otimes I}^2 \\
		&\quad= F(x^k) - \norm{\mathbf{B}z^k - \mathbf{B}z^\star}_{\frac{1}{2}S_L \otimes I}^2
	\end{align*}
	where $S_L = \Exp[Q_{i^k}^T G + G^T Q_{i^k} - L Q_{i^k}^T Q_{i^k}]$.

	With $D = \begin{bmatrix} 0 & \mathbf{1}^T \end{bmatrix}$ we have  $(K\otimes I)z^k = x^k + \frac{\lambda}{n}(D\otimes I)\mathbf{B}z^k$.
	Using the first order convexity condition on $F$ and $0 = (D\otimes I)\mathbf{B}z^\star = (G\otimes I)\mathbf{B}z^\star$ yields
	\begin{equation}\label{eq:fvalcomb-cvx}
		\begin{aligned}
			F((K\otimes I)z^k)
			&= F(x^k + \tfrac{\lambda}{n}(D \otimes I)\mathbf{B}z^k) \\
			&\geq F(x^k) + \inprod{\nabla F(x^k)}{ \tfrac{\lambda}{n}(D \otimes I)\mathbf{B}z^k} \\
			&= F(x^k) + \inprod{(G\otimes I) \mathbf{B}z^k}{ \tfrac{\lambda}{n}(D \otimes I)\mathbf{B}z^k} \\
			&= F(x^k) + \norm{\mathbf{B}z^k}_{\frac{1}{2}\frac{\lambda}{n}(D^TG + G^TD)\otimes I}^2 \\
			&= F(x^k) + \norm{\mathbf{B}z^k - \mathbf{B}z^\star}_{\frac{1}{2}S_C \otimes I}^2
		\end{aligned}
	\end{equation}
	where $S_C = \frac{\lambda}{n}(D^TG + G^TD)$.
	Combining these two inequalities gives
	\begin{align*}
		\Exp[F((K\otimes I) z^{k+1})|\mathcal{F}_k] \leq F((K\otimes I) z^k) - \norm{\mathbf{B}z^k - \mathbf{B}z^\star}^2_{\frac{1}{2}S\otimes I}
	\end{align*}
	where  $S = S_L + S_C$.
\end{proof}

\subsection{Convergence Theorems}\label{sec:thm}

We are now ready to state the main convergence theorems for SVAG.
They are stated with the notation from \algref{alg:svag} but are proved at the end of this section with the help of the reformulation in \eqref{eq:zupdate} and the lemmas above.
\begin{thm}\label{thm:conv-coco}
	For all $i\in\{1,\dots,n\}$, let $(x^k)_{k\in\N}$ and $(y_i^k)_{k\in\N}$ be the sequences generated by \algref{alg:svag}.
	If \assref{ass:op} hold and the step-size, $\lambda > 0$, and innovation weight, $\theta \in\R$, satisfy
	\begin{align*}
		\frac{1}{L(2 + |n - \theta|)} > \lambda,
	\end{align*}
	then $x^k \rightarrow x^\star$ and $y_i^k \rightarrow R_i x^\star$ almost surely for all $i\in\{1,\dots,n\}$, where $x^\star$ is a solution to \eqref{eq:zerprob}.
	For all $i\in\{1,\dots,n\}$, the residuals converge a.s.\ as
	\begin{align*}
		\min_{k \in \{0,\dots,t\}} \Exp[\norm{R_ix^k - R_ix^\star}^2] &\leq \tfrac{n}{\lambda(L^{-1} - \lambda c)} \tfrac{1}{t+1} C_R, \\
		\min_{k \in \{0,\dots,t\}} \Exp[\norm{y_i^k - R_ix^\star}^2] &\leq \tfrac{n}{\lambda(L^{-1} - \lambda c)} \tfrac{1}{t+1} C_R
	\end{align*}
	where $c = 2 + |n-\theta|$ and
	\begin{align*}
		C_R &= \min_{x\in X^\star} \norm{x^0 - x}^2 + \tfrac{\lambda}{L}{\textstyle\sum_{i=1}^n}\norm{y_i^0 - R_ix^\star}^2 + \lambda^2(n-\theta)^2\norm{\tfrac{1}{n}{\textstyle\sum_{i=1}^n}y_i^0}^2 \\
		&\qquad\quad\;\; - 2\lambda(n-\theta)\inprod{x^0-x}{\tfrac{1}{n}{\textstyle\sum_{i=1}^n}y_i^0 }
	\end{align*}
	for any $x^\star \in X^\star$.
\end{thm}

\begin{thm}\label{thm:conv-func}
	For all $i\in\{1,\dots,n\}$, let $(x^k)_{k\in\N}$ and $(y_i^k)_{k\in\N}$ be the sequences generated by \algref{alg:svag}.
	If \assref{ass:grad} hold and the step-size, $\lambda > 0$, and innovation weight, $\theta \in [0,n]$, satisfy
	\begin{align*}
		\frac{1}{L}\frac{1}{2 + (n-\theta)\tfrac{\theta-1}{n}\big(\tfrac{\theta - 1}{n} - 1 + \tfrac{\theta - 1}{|\theta - 1|} \sqrt{2}\big)} > \lambda,
	\end{align*}
	then $x^k \rightarrow x^\star$ and $y_i^k \rightarrow \nabla f_i(x^\star)$ almost surely, where $x^\star$ is a solution to \eqref{eq:zerprob}.
	For all $i\in\{1,\dots,n\}$, the residuals converge a.s.\ as
	\begin{align*}
		\min_{k \in \{0,\dots,t\}} \Exp[\norm{\nabla f_i(x^k) - \nabla f_i(x^\star)}^2] &\leq \tfrac{n}{\lambda(L^{-1} - \lambda c)} \tfrac{1}{t+1} (C_R + C_F), \\
		\min_{k \in \{0,\dots,t\}} \Exp[\norm{y_i^k - \nabla f_i(x^\star)}^2] &\leq \tfrac{n}{\lambda(L^{-1} - \lambda c)} \tfrac{1}{t+1} (C_R + C_F), \\
		\min_{k \in \{0,\dots,t\}} \Exp[F(x^k) - F(x^\star)] &\leq \tfrac{1}{\lambda(1 - L\lambda c)} \tfrac{1}{t+1} (C_R + C_F)
	\end{align*}
	where
	\begin{align*}
		c &= 2 + (n-\theta)\tfrac{\theta-1}{n}\big(\tfrac{\theta - 1}{n} - 1 + \tfrac{\theta-1}{|\theta-1|} \sqrt{2}\big), \\
		C_R &= \min_{x\in X^\star} \norm{x^0 - x}^2 + \tfrac{\lambda}{L}{\textstyle\sum_{i=1}^n}\norm{y_i^0 - R_ix^\star}^2 + \lambda^2(n-\theta)^2\norm{\tfrac{1}{n}{\textstyle\sum_{i=1}^n}y_i^0}^2 \\
		&\qquad\quad\;\; - 2\lambda(n-\theta)\inprod{x^0-x}{\tfrac{1}{n}{\textstyle\sum_{i=1}^n}y_i^0 }, \\
		C_F &= 2\lambda(n-\theta)\big(F(x^0 + \tfrac{\lambda}{n}{\textstyle\sum_{i=1}^n} y_i^0) - F(x^\star) \big)
	\end{align*}
	for any $x^\star \in X^\star$.
\end{thm}

Both \thmref{thm:conv-coco} and \ref{thm:conv-func} give the step-size condition $\lambda\in(0,\tfrac{1}{2L})$ for the SAGA special case, i.e.,  $\theta =n$.
This is the same as the largest upper bound found in the literature \cite{defazio_saga_2014} and appears to be tight \cite{morinSamplingUpdateFrequencies2020}.
\thmref{thm:conv-func} also give this step-size condition when $\theta = 1$, i.e., SAG in the optimization case.
This bound improves on upper bound of $\frac{1}{16L} \leq \lambda$ presented in \cite{schmidt_minimizing_2017}.

In the cocoercive operator setting with $\theta \neq n$, \thmref{thm:conv-coco} gives a step-size condition that scales with $n^{-1}$.
This step-size scaling is significantly worse compared to the gradient case in \thmref{thm:conv-func} in which the step-size's dependence on $n$ is $\mathcal{O}(1)$ for all $\theta$.
This difference is indeed real and not an artifact of the analysis since we in \secref{sec:num-exp} present a problem for which the cocoercivity result appears to be tight.
A consequence of this unfavorable step-size scaling in the operator setting is slow convergence.
There is therefore little reason to use anything else than $\theta=n$ in SVAG when $R_i$ is not a gradient of a smooth function for all $i\in\{1,\dots,n\}$.

The rates of \thmref{thm:conv-coco} and \ref{thm:conv-func} are of $\mathcal{O}(\frac{1}{t+1})$ type with two sets of multiplicative factors.
One factor which only depend on the algorithm parameters, $\frac{n}{\lambda(L^{-1} - \lambda c)}$, and one set which depend on how the algorithm initialization relates to the solution set, $C_R$ and $C_R + C_F$.
The initialization dependent factors also depend on the algorithm parameters, but, since knowing the exact dependency requires knowing the solution set, we will not attempt to tune the parameters to decrease this factor.
Only considering the first factor, the rate becomes better if $c$ is decreased and, since $c$ is independent of $\lambda$, the best choice of step-size is $\lambda = (2Lc)^{-1}$.
This means that $\lambda = (4L)^{-1}$ and $\theta = n$ are the best parameter choices in the cocoercive operator setting.
In the optimization case the best step-size is also $\lambda = (4L)^{-1}$ but the innovation weight can be selected as either $\theta=n$ or $\theta=1$.

However, in the optimization case we do not believe that these theoretical rates reflects real world performance and parameter choices based on them might therefore not perform particularly well.
We base this belief on our experience with numerical experiments.
For $\theta \neq n$ and $\theta \neq 1$, we have not found any optimization problem where the step-size condition in \thmref{thm:conv-func} appears to be tight.
Also, using $\lambda=(2Lc)^{-1}$ as suggested by the \thmref{thm:conv-func} can in some cases lead to impractically small step-sizes.
For instance, if $\lambda =(2Lc)^{-1}$ was used in the experiments in \secref{sec:num-exp}, a couple of the experiments would have step-sizes over 1000 times smaller than the ones used now.
One can of course not disprove a worst case analysis with experiments but we still feel they indicate a conservative analysis, even though the analysis improves on the previous best results.

\begin{proof}[Proof of \thmref{thm:conv-coco}]
	Apply \lemref{lem:fejer-cond} with $\xi = 0$, the iterates given by \eqref{eq:zupdate} then satisfy the following for all $z^\star \in Z^\star$,
	\begin{equation}\label{eq:coco-lyapunov}
		\begin{aligned}
			&\Exp[\norm{z^{k+1} - z^\star}_{H\otimes I}^2 | \mathcal{F}_k] \\
			&\qquad\leq \norm{z^k - z^\star}_{H\otimes I}^2 - \norm{\mathbf{B}z^k - \mathbf{B}z^\star}_{(2 M  - \Exp[U_{i^k}^THU_{i^k}]) \otimes I }^2.
		\end{aligned}
	\end{equation}
	Assuming $H \succ 0$ and $2M - \Exp [U_{i^k}^T H U_{i^k}] \succ 0$ can \propref{prop:supermartingale} be applied.
	We will later prove that this assumption indeed does hold.
	\propref{prop:supermartingale} gives a.s.\ summability of $\norm{\mathbf{B}z^k - \mathbf{B}z^\star}_{(2 M - \Exp [U_{i^k}^THU_{i^k}]) \otimes I }^2$ and hence will $\mathbf{B}z^k \to \mathbf{B}z^\star$ almost surely.
	\lemref{lem:clusterpoints} then gives that all cluster points of $(z^k)_{k\in\N}$ are in $Z^\star$ almost surely.
	Finally, since \propref{prop:supermartingale} ensures the a.s.\ convergence of $\norm{z^k - z^\star}_{H\otimes I}^2$ and since $\R^{N(n+1)}$ with the inner product $\inprod{(H\otimes I)\cdot}{\cdot}$ is a finite dimensional Hilbert space, \propref{prop:stochseqconv} gives the almost sure convergence of $z^k\to z^\star \in Z^\star$.

	There always exists a $\lambda$ such that $2M - \Exp [U_{i^k}^T H U_{i^k}]$ and $H$ are positive definite.
	First we show that $H \succ 0$ always holds for $\lambda > 0$.
	Taking the Schur complement of $1$ in $H$ gives
	\begin{align*}
		\frac{\lambda}{L}I + \frac{\lambda^2}{n^2}(n-\theta)^2 \mathbf{11}^T - \frac{\lambda^2}{n^2}(n-\theta)^2\mathbf{11}^T = \frac{\lambda}{L}I \succ 0.
	\end{align*}
	Hence is $H \succ 0$ since the Schur complement is positive definite.

	We now show $2M - \Exp [U_{i^k}^T H U_{i^k}] \succ 0$.
	Straightforward algebra, see the supplementary material, yields
	\begin{align*}
		2M - \Exp [U_{i^k}^T H U_{i^k}]
		&=
		\begin{bmatrix}
			1 & 0 \\
			0 & 1
		\end{bmatrix} \otimes \frac{\lambda}{nL} I
		-
		\begin{bmatrix}
			1 & 0 \\
			0 & 1
		\end{bmatrix} \otimes \frac{\lambda^2}{n} I
		+
		\begin{bmatrix}
			0 & 1 \\
			1 & 0
		\end{bmatrix} \otimes \frac{\lambda^2}{n}(I - \frac{1}{n}\mathbf{11}^T)
		\\
		&\quad
		-
		\begin{bmatrix}
			0 & 1\\
			1 & 0
		\end{bmatrix} \otimes (n-\theta)\frac{\lambda^2}{n^2}\mathbf{11}^T
		+
		\begin{bmatrix}
			0 & 0 \\
			0 & 1
		\end{bmatrix} \otimes \frac{\lambda^2}{n^2}\mathbf{11}^T.
	\end{align*}
	Positive definiteness of this matrix is established by ensuring positivity of the smallest eigenvalue $\sigma_{\min}$.
	The smallest eigenvalue $\sigma_{\min}$ is greater than the sum of the smallest eigenvalue of each term.
	For the eigenvalues of the Kronecker products, see \eqref{eq:kroneig}.
	This gives that
	\begin{align*}
		\sigma_{\min}
		\geq \frac{\lambda}{nL} - \frac{\lambda^2}{n} - \frac{\lambda^2}{n} - \frac{\lambda^2}{n}|n-\theta| + 0
		= \frac{\lambda}{n}\big(L^{-1} - \lambda(2 + |n-\theta|)\big).
	\end{align*}
	Since $\lambda > 0$ by assumption, if
	\begin{align*}
		\frac{1}{L(2 + |n - \theta|)} > \lambda.
	\end{align*}
	we have that $\sigma_{\min} > 0$ and that $2M - \Exp[U_{i^k}^T H U_{i^k}]$ is positiv definite.

	Rates are gotten by taking the total expectation of \eqref{eq:coco-lyapunov} and adding together the inequalities from $k=0$ to $k=t$, yielding
	\begin{align*}
		\norm{z^0 - z^\star}_{H\otimes I}^2 =
		&\Exp[\norm{z^0 - z^\star}_{H\otimes I}^2] - \Exp[\norm{z^{t+1} - z^\star}_{H\otimes I}^2] \\
		&\geq {\textstyle\sum_{k=0}^t} \Exp\big[ \norm{\mathbf{B}z^k - \mathbf{B}z^\star}_{(2 M  - \Exp[U_{i^k}^THU_{i^k}]) \otimes I }^2 \big] \\
		&\geq {\textstyle\sum_{k=0}^t} \sigma_{\min} \Exp[\norm{\mathbf{B}z^k - \mathbf{B}z^\star}^2 ] \\
		&\geq \sigma_{\min}(t+1) \min_{k\in\{0,\dots,t\}} \Exp[\norm{\mathbf{B}z^k - \mathbf{B}z^\star}^2].
	\end{align*}
	Putting in the lower bound on $\sigma_{\min}$ and rearranging yields
	\begin{align*}
		\min_{k\in\{0,\dots,t\}} \Exp[\norm{\mathbf{B}z^k - \mathbf{B}z^\star}^2]
		\leq
		\tfrac{n}{\lambda(L^{-1} - \lambda(2 + |n-\theta|))(t+1)}\norm{z^0 - z^\star}_{H\otimes I}^2.
	\end{align*}
	From the definition of $H$ in \lemref{lem:fejer-cond} we have
	\begin{align*}
		\norm{z^0 - z^\star}_{H\otimes I}^2
		&= \norm{x^0 - x^\star}^2 + \tfrac{\lambda}{L}{\textstyle\sum_{i=1}^n}\norm{y_i^0 - R_ix^\star}^2 + \lambda^2(n-\theta)^2\norm{\tfrac{1}{n}{\textstyle\sum_{i=1}^n}y_i^0}^2 \\
		&\quad - 2\lambda(n-\theta)\inprod{x^0-x^\star}{\tfrac{1}{n}{\textstyle\sum_{i=1}^n}y_i^0 }
	\end{align*}
	where $z^\star = (x^\star,R_1x^\star,\dots,R_nx^\star)$.
	Since this hold for any $z^\star\in Z^\star$ and hence any $x^\star\in X^\star$, the results of theorems follows by minimizing the RHS over $x^\star \in X^\star$.
	Note, since $R_ix^\star$ constant for all $x^\star \in X^\star$, the objective is convex and, since $X^\star$ is closed and convex, the minimum is then attained.
\end{proof}

\begin{proof}[Proof of \thmref{thm:conv-func}]
	Combining \lemref{lem:fejer-cond} and \ref{lem:funcdecrease} yield
	\begin{align*}
		\Exp&[\norm{z^{k+1} - z^\star}_{H\otimes I}^2 + 2\lambda (n-\theta)(F((K\otimes I)z^{k+1}) - F(x^\star))| \mathcal{F}_k] \\
		&\quad\leq \norm{z^k - z^\star}_{H\otimes I}^2 + 2\lambda (n-\theta)(F((K\otimes I)z^{k}) - F(x^\star))\\
		&\quad\quad- \norm{\mathbf{B}z^k - \mathbf{B}z^\star}_{(2 M - \Exp [U_{i^k}^THU_{i^k}] + \lambda (n-\theta)S - \xi I) \otimes I }^2  - {\xi n L}\inprod{\nabla F(x^k) }{x^k - x^\star}
	\end{align*}
	which holds for all $k\in\N$, $\xi \in [0, \frac{2\lambda}{nL}]$, and $z^\star \in Z^\star$.
	Since $H \succ 0$ for $\lambda > 0$, see the proof of \thmref{thm:conv-coco}, the first term is non-negative while the second term is non-negative if $\theta \leq n$.
	From cocoercivity of $\nabla F$, the last term is non-positive and we assume, for now, that there exists $\lambda > 0$ and $\frac{2\lambda}{nL}\geq \xi > 0$ such that $2 M - \Exp [U_{i^k}^THU_{i^k}] + \lambda (n-\theta) S - \xi I \succ 0$, making the third term non-positive.

	Applying \propref{prop:supermartingale} gives the a.s.\ summability of
	\begin{align*}
		\norm{\mathbf{B}z^k - \mathbf{B}z^\star}_{(2 M - \Exp [U_{i^k}^THU_{i^k}] + \lambda (n-\theta)S - \xi I) \otimes I }^2  + {\xi n L}\inprod{\nabla F(x^k) - \nabla F(x^\star)}{x^k - x^\star}.
	\end{align*}
	Since both term are positive, both terms are a.s.\ summable.
	From the first term we have the a.s.\ convergence of $\mathbf{B}z^k \to  \mathbf{B}z^\star$ and	\lemref{lem:clusterpoints} then gives that all cluster points of $(z^k)_{k\in\N}$ are almost surely in $Z^\star$.
	For the second term we note that by convexity we have
	\begin{align*}
		\inprod{\nabla F(x^k) - \nabla F(x^\star)}{x^k - x^\star} \geq F(x^k) - F(x^\star) \geq 0
	\end{align*}
	and $F(x^k) - F(x^\star)$ then is summable a.s.\ since $\xi nL > 0$.
	Using smoothness of $F$, \eqref{eq:descent} and the notation from \eqref{eq:fvalcomb-cvx} gives
	\begin{align*}
		&F(x^\star) \leq F((K\otimes I)z^k) \\
		&\qquad = F(x^k + \tfrac{\lambda}{n}(D\otimes I)\mathbf{B}z^k) \\
		&\qquad \leq F(x^k) + \inprod{(G\otimes I) \mathbf{B}z^k}{\tfrac{\lambda}{n}(D\otimes I) \mathbf{B}z^k} + \tfrac{L}{2}\norm{\tfrac{\lambda}{n}(D\otimes I)\mathbf{B}z^k}^2 \\
		&\qquad \leq F(x^k) + \norm{(G\otimes I) \mathbf{B}z^k}\norm{\tfrac{\lambda}{n}(D\otimes I) \mathbf{B}z^k} + \tfrac{L}{2}\norm{\tfrac{\lambda}{n}(D\otimes I)\mathbf{B}z^k}^2 \\
		&\qquad \to F(x^\star) \text{ a.s.}
	\end{align*}
	since $(G\otimes I)\mathbf{B}z^k \to (G\otimes I)\mathbf{B}z^\star = 0$ and $(D\otimes I)\mathbf{B}z^k \to (D\otimes I)\mathbf{B}z^\star = 0$ almost surely.
	Therefore we have the a.s.\ convergence of $F((K\otimes I)z^k) - F(x^\star) \to 0$.

	From \propref{prop:supermartingale} we can also conclude that $\norm{z^k - z^\star}_{H\otimes I}^2 + 2\lambda (n-\theta)(F((K\otimes I)z^{k}) - F(x^\star))$ a.s.\ converge to a non-negative random variable.
	Since $F((K\otimes I)z^k) - F(x^\star) \to 0$ a.s.\ we have that $\norm{z^k - z^\star}_{H\otimes I}^2$ also must a.s.\ converge to a non-negative random variable.
	\propref{prop:stochseqconv} then give the almost sure convergence of $(z^k)_{k\in\N}$ to $Z^\star$.

	We now show that there exists $\lambda > 0$ and $\xi > 0$ such that
	\begin{align*}
		&2M - \Exp [U_{i^k}^T H U_{i^k}] + \lambda (n-\theta) S - \xi I \\
		&\quad=
		\begin{bmatrix}
			1 & 0 \\
			0 & 1
		\end{bmatrix} \otimes \frac{\lambda}{nL} I
		-
		\begin{bmatrix}
			1 & 0 \\
			0 & 1
		\end{bmatrix} \otimes \frac{\lambda^2}{n} I
		+
		\begin{bmatrix}
			0 & 1 \\
			1 & 0
		\end{bmatrix} \otimes \frac{\lambda^2}{n}(I - \frac{1}{n}\mathbf{11}^T)
		+
		\begin{bmatrix}
			0 & 0 \\
			0 & 1
		\end{bmatrix} \otimes \frac{\lambda^2}{n^2}\mathbf{11}^T
		\\
		&\quad\quad +
		\begin{bmatrix}
			2 & -1 \\
			-1 & 0
		\end{bmatrix} \otimes  (n-\theta)(\theta - 1)\frac{\lambda^2}{n^3} \mathbf{11}^T
		-
		\begin{bmatrix}
			1  & -1  \\
			-1  & 1
		\end{bmatrix} \otimes  (n-\theta)(\theta - 1)^2\frac{L\lambda^3}{n^3}  I
		\\
		&\quad\quad-
		\begin{bmatrix}
			1 & 0\\
			0 & 1
		\end{bmatrix} \otimes \xi I \succ 0.
	\end{align*}
	We show positive definiteness by ensuring that the smallest eigenvalue is positive.
	The smallest eigenvalue $\sigma_{\min}$ is greater than the sum of the smallest eigenvalues of each term,
	\begin{align*}
		\sigma_{\min}
		\geq &\frac{\lambda}{nL} - \frac{\lambda^2}{n} - \frac{\lambda^2}{n} + 0 + (1 - \tfrac{\theta-1}{|\theta-1|} \sqrt{2})(n-\theta)(\theta - 1)\frac{\lambda^2}{n^2} \\
		&- 2(n-\theta)(\theta - 1)^2 \frac{L\lambda^3}{n^3} - \xi.
	\end{align*}
	Assuming $\lambda \leq \frac{1}{2L}$ yields the following lower bound on the smallest eigenvalue
	\begin{align*}
		\sigma_{\min}
		&\geq \frac{\lambda}{nL} - \frac{2\lambda^2}{n} + (1- \tfrac{\theta-1}{|\theta-1|} \sqrt{2})(n-\theta)(\theta-1)\frac{\lambda^2}{n^2} - (n-\theta)(\theta - 1)^2 \frac{\lambda^2}{n^3} - \xi
		\\
		&= \frac{\lambda}{n}\big(L^{-1} - \lambda\big(2 + (n-\theta)\tfrac{\theta-1}{n}\big(\tfrac{\theta - 1}{n} - 1 + \tfrac{\theta-1}{|\theta-1|} \sqrt{2}\big)\big)\big) - \xi.
	\end{align*}
	Selecting
	\begin{align*}
		\xi = \frac{\lambda}{2n}\big(L^{-1} - \lambda\big(2 + (n-\theta)\tfrac{\theta-1}{n}\big(\tfrac{\theta - 1}{n} - 1 + \tfrac{\theta-1}{|\theta-1|} \sqrt{2}\big)\big)\big),
	\end{align*}
	which satisfy the assumption $\frac{2\lambda}{nL}\geq \xi > 0$, yield $\sigma_{\min} \geq \xi$.
	Since $\lambda > 0$ by assumption, if
	\begin{align*}
		\frac{1}{L} \frac{1}{2 + (n-\theta)\tfrac{\theta-1}{n}(\tfrac{\theta - 1}{n} - 1 + \tfrac{\theta-1}{|\theta-1|} \sqrt{2})}  > \lambda
	\end{align*}
	we have that $\sigma_{\min} \geq \xi > 0$ and hence that the examined matrix is positive definite.
	Furthermore, if $\lambda$ satisfies the above inequality it also satisfies the assumption $\lambda \leq \frac{1}{2L}$.

	Rates are gotten in the same way as for \thmref{thm:conv-coco}, the total expectation is taken of the Lyapunov inequality at the beginning of the proof and the inequalities are summed from $k=0$ to $k=t$.
	\begin{align*}
		&\norm{z^0 - z^\star}_{H\otimes I}^2 + 2\lambda (n-\theta)(F((K\otimes I)z^{0}) - F(x^\star))\\
		&\quad\geq {\textstyle\sum_{k=0}^t}\big(\sigma_{\min} \Exp[\norm{\mathbf{B}z^k - \mathbf{B}z^\star}^2]  + \Exp[\sigma_{\min} n L\inprod{\nabla F(x^k) }{x^k - x^\star}]\big) \\
		&\quad\geq \sigma_{\min}(t+1) \min_{k\in\{1,\dots,t\}}\big(\Exp[\norm{\mathbf{B}z^k - \mathbf{B}z^\star}^2] + \Exp[ n L\inprod{\nabla F(x^k) }{x^k - x^\star}]\big) \\
		&\quad\geq \sigma_{\min}(t+1) \min_{k\in\{1,\dots,t\}}\big(\Exp[\norm{\mathbf{B}z^k - \mathbf{B}z^\star}^2] + nL\Exp[F(x^k) - F(x^\star)]\big).
	\end{align*}
	Inserting the lower bound on $\sigma_{\min}$, rearranging and minimizing over $x^\star\in X^\star$ yield the results of the theorem.
\end{proof}

\section{Numerical Experiments}\label{sec:num-exp}
A number of experiments, outlined below, were performed to verify the tightness of the theory in the cocoercive operator case and examine the effect of bias in the cocoercive gradient case.
The experiments were implemented in \texttt{Julia} \cite{bezansonJuliaFreshApproach2017} and, together with several other VR-SG methods, can be found at \url{https://github.com/mvmorin/VarianceReducedSG.jl}.

\subsection{Cocoercive Operators Case}
\begin{figure*}[t]
	\begin{minipage}[t]{.49\textwidth}
		\centering
		\includegraphics[width=\textwidth]{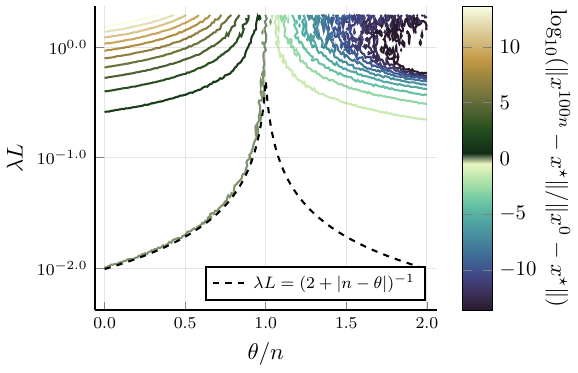}
		{(a) $n=100$}
	\end{minipage}
	\begin{minipage}[t]{.49\textwidth}
		\centering
		\includegraphics[width=\textwidth]{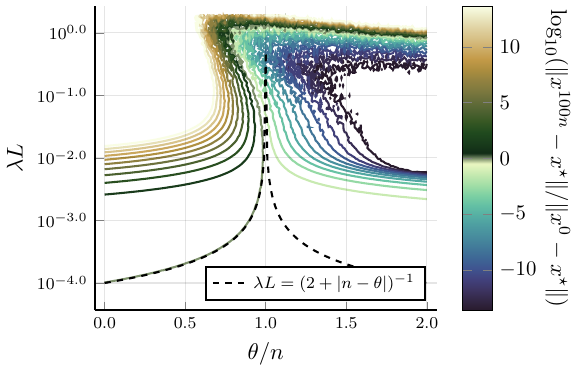}
		{(b) $n=10000$}
	\end{minipage}
	\caption{
		Root-finding of Averaged Rotations:
		Relative distance to the solution after $100n$ iterations of SVAG together the step-size upper bound, $\lambda L < (2 + |n - \theta|)^{-1}$.
		Note how well the $0$th level, i.e., the boundary between convergence and divergence, follow the upper bound on the step-size.
	}\label{fig:averagedrotation}
\end{figure*}

In order for the difference between cocoercive operators and cocoercive gradients to not be an artifact of our analysis, the results in the operator case can not be overly conservative.
We therefore construct a cocoercive operator problem for which the results appear to be tight, thereby verifying the difference.
Consider problem~\eqref{eq:zerprob} where the operator $R_i : \R^2 \to \R^2$ is an averaged rotation
\begin{align*}
	R_i =
	\frac{1}{2}
	\begin{bigmatrix}
		1 & 0 \\
		0 & 1
	\end{bigmatrix}
	+
	\frac{1}{2}
	\begin{bigmatrix}
		\cos\tau & -\sin\tau\\
		\sin\tau & \cos\tau
	\end{bigmatrix}
\end{align*}
for all $i \in \{1,\dots,n\}$ and some $\tau\in[0, 2\pi)$.
The operators are $1$-cocoercive and the zero vector is the only solution to \eqref{eq:zerprob} if $\tau\neq\pi$.
The step-size condition from \thmref{thm:conv-coco} appears to be tight for $\theta\in[0,n]$ when the angle of rotation $\tau$ approaches $\pi$.
We therefore let $\tau = \frac{179}{180}\pi$ and solve the problem with different configurations of step-size $\lambda$ and innovation weight $\theta$.

Figure \ref{fig:averagedrotation} displays the relative distance to the solution after $100n$ iterations of SVAG together with the upper bound on the step-size.
When $\theta\in[0,n]$ and $\lambda$ exceeds the upper bound, the distance to the solution increases for both $n=100$ and $n=10000$, i.e., the method does not converge.
Hence, for $\theta\in[0,n]$, the step-size bound in \thmref{thm:conv-coco} appears to be tight.
However, it is noteworthy that for this particular problem it seems beneficial to exceed the step-size bound when $\theta > n$.

\subsection{Cocoercive Gradients Case}
Since, as we stated in \secref{sec:thm}, we do not believe that the theoretical rates are particularly tight in the optimization case, we examine the effects of the bias numerically.
These experiments can of course not be exhaustive and we choose to focus on only the bias parameter $\theta$ and perform all experiments with the same step-size.
This also demonstrate why we believe the analysis to be conservative since the chosen step-size in some cases are a 1000 times larger than upper bound from \thmref{thm:conv-func}.
Convergence with this large of a step-size have also been seen elsewhere with both \cite{schmidt_minimizing_2017} and \cite{driggsBiasedStochasticGradient2020} disregarding their own the theoretical step-size conditions.

\begin{figure*}[t]
	\begin{minipage}[t]{.32\textwidth}
		\centering
		\includegraphics[width=\textwidth]{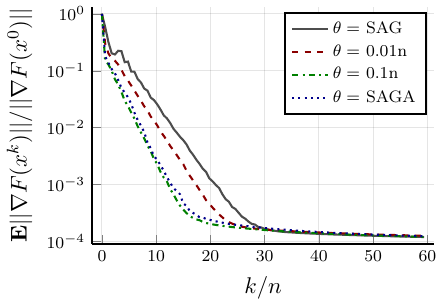}
		{(a) \texttt{protein}}
	\end{minipage}
	\begin{minipage}[t]{.32\textwidth}
		\centering
		\includegraphics[width=\textwidth]{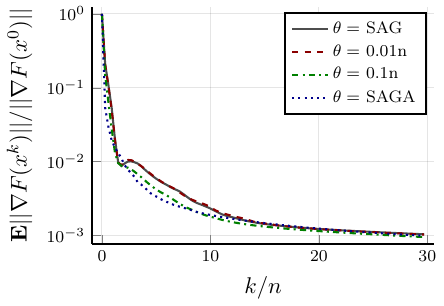}
		{(b) \texttt{breast-cancer\_scale}}
	\end{minipage}
	\begin{minipage}[t]{.32\textwidth}
		\centering
		\includegraphics[width=\textwidth]{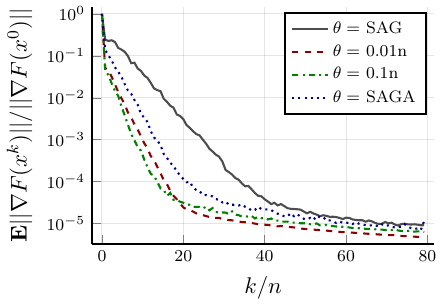}
		{(c) \texttt{a9a}}
	\end{minipage}\\
	\begin{minipage}[t]{.32\textwidth}
		\centering
		\includegraphics[width=\textwidth]{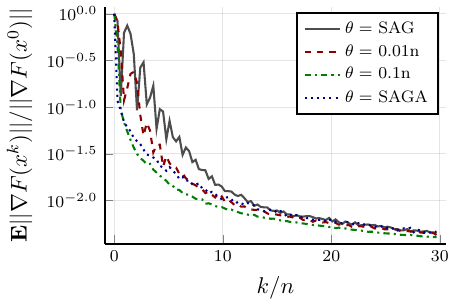}
		{(d) \texttt{gisette\_scale}}
	\end{minipage}
	\begin{minipage}[t]{.32\textwidth}
		\centering
		\includegraphics[width=\textwidth]{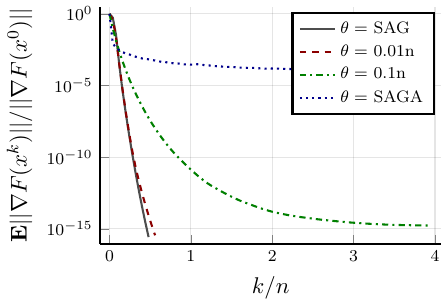}
		{(e) \texttt{mushrooms}}
	\end{minipage}
	\begin{minipage}[t]{.32\textwidth}
		\centering
		\includegraphics[width=\textwidth]{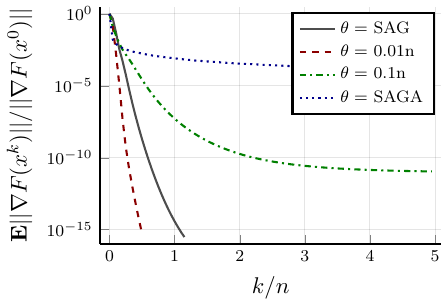}
		{(f) \texttt{mnist.scale}}
	\end{minipage}
	\caption{
		Logistic Regression: Expected gradient norm for each iteration.
		The expected value is estimated with the sample average of 100 runs.
		A step-size of $\lambda = \frac{1}{2L}$ was used in all cases.
	}\label{fig:logreg}
\end{figure*}

The experiments are done by performing a rough parameter sweep over the innovation weight $\theta$ on two different binary classification problems and we will look for patterns in how the convergence is affected.
The first problem is logistic regression,
\begin{align*}
	\min_{x} \tfrac{1}{n}\sum_{i=1}^n \log(1 + e^{-y_i a_i^T x}).
\end{align*}
The second is SVM with a square hinge loss,
\begin{align*}
	\min_x \tfrac{1}{n}\sum_{i=1}^n\big(\max(0,1-y_i a_i^T x)^2 + \tfrac{\gamma}{2}\norm{x}^2\big)
\end{align*}
where $\gamma > 0$ is a regularization parameter.
In both problems are $y_i \in \{-1,1\}$ the label and $a_i \in \R^N$ the features of the $i$th training data point.
Note, although not initially obvious, $\max(0,\cdot)^2$ is convex and differentiable with Lipschitz continuous derivative and the second problem is therefore indeed smooth.
The logistic regression problem does not necessarily have a unique solution and the distance to the solution set is therefore hard to estimate.
For this reason, we examine the convergence of $\|\nabla F(x^k)\|\to 0$ instead of the distance to the solution set.

The datasets for both these classification problems are taken from the \texttt{LibSVM}\cite{changLIBSVMLibrarySupport2011} collection of datasets.
The number of examples in the datasets varies between $n = 683$ and $n = 60,000$ while the number of features is between $N = 10$ and $N = 5,000$.
Two of the datasets, \texttt{mnist.scale} and \texttt{protein}, consist of more than 2 classes.
These are converted to binary classification problems by grouping the different classes into two groups.
For the digit classification dataset \texttt{mnist.scale}, the digits are divided into the groups 0-4 and 5-9.
For the \texttt{protein} data set, the classes are grouped as 0 and 1-2.
The results of solving the classification problems above can be found in Figures \ref{fig:logreg} and \ref{fig:sqrhingesvm}.

\begin{figure*}[t]
	\begin{minipage}[t]{.32\textwidth}
		\centering
		\includegraphics[width=\textwidth]{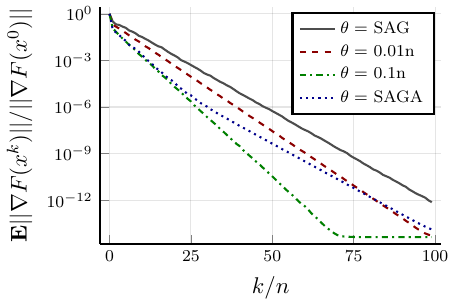}
		{(a) \texttt{protein}, $\gamma = 10^{-3}$}
	\end{minipage}
	\begin{minipage}[t]{.32\textwidth}
		\centering
		\includegraphics[width=\textwidth]{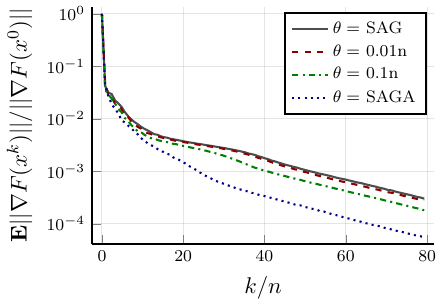}
		{(b) \centering \texttt{breast-cancer\_scale}, $\gamma = 10^{-3}$}
	\end{minipage}
	\begin{minipage}[t]{.32\textwidth}
		\centering
		\includegraphics[width=\textwidth]{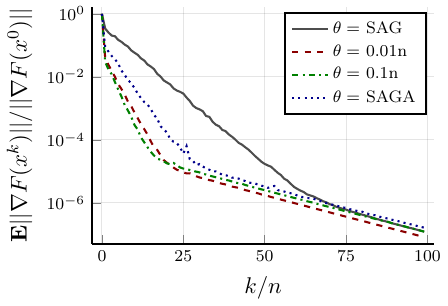}
		{(c) \texttt{a9a}, $\gamma = 10^{-4}$}
	\end{minipage}\\
	\begin{minipage}[t]{.32\textwidth}
		\centering
		\includegraphics[width=\textwidth]{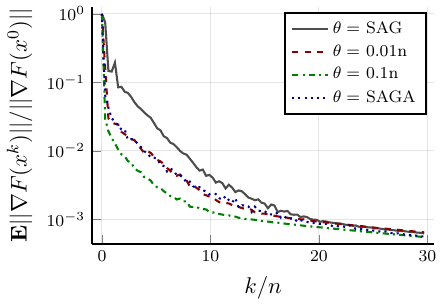}
		{(d) \texttt{gisette\_scale}, $\gamma = 10^{-1}$}
	\end{minipage}
	\begin{minipage}[t]{.32\textwidth}
		\centering
		\includegraphics[width=\textwidth]{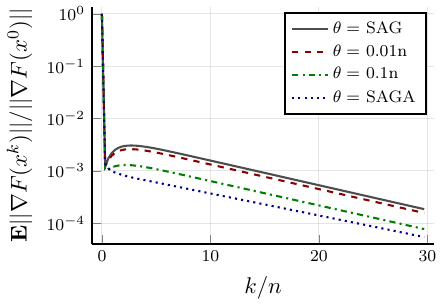}
		{(e) \texttt{mushrooms}, $\gamma = 10^{-3}$}
	\end{minipage}
	\begin{minipage}[t]{.32\textwidth}
		\centering
		\includegraphics[width=\textwidth]{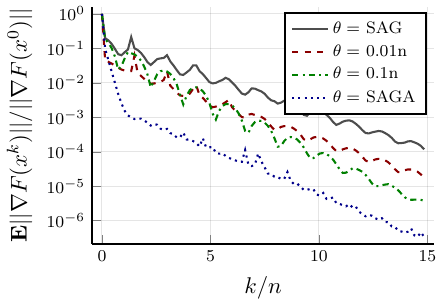}
		{(f) \texttt{mnist.scale}, $\gamma = 10^{-1}$}
	\end{minipage}
	\caption{
		Square Hinge Loss SVM: Expected gradient norm for each iteration.
		The expected value is estimated with the sample average of 100 runs.
		A step-size of $\lambda = \frac{1}{2L}$ was used in all cases.
	}\label{fig:sqrhingesvm}
\end{figure*}

From Figures \ref{fig:logreg} and \ref{fig:sqrhingesvm} it appears like the biggest difference between the innovation weights are in the early stages of the convergence.
Most innovation weight choices appear to eventually converge with the same rate.
In the cases where this does not happen, the fastest converging choice of innovation weight actually reaches machine precision.
It is therefore not possible to say whether these cases would eventually reach the same rate as well.
Since none of the choices of $\theta$ appears to consistently be at a significant disadvantage, even though the step-size used exceeds the upper bound in \thmref{thm:conv-func} when $\theta= 0.1n$ and $\theta = 0.01n$, we conjecture that the asymptotic rates for a given step-size is independent of $\theta$.

The initial phase can clearly have a large impact on the convergence and it can therefore still be a benefit to tuning the bias.
However, comparing the different choices of innovation weight yields no clear conclusion since no single choice of innovation weight consistently outperforms another.
In most cases do the lower bias choices---$\theta = n$ (SAGA) or $\theta = 0.1n$---seem perform best but, when they do not, the high bias choices---$\theta = 1$ (SAG) and $\theta=0.01n$---perform significantly better.
Another observation is that lowering $\theta$ increases any oscillations if they are present.
We speculate that it is due to the increased inertia and we also believe that this inertia is what allows the lower innovation weights to sometimes perform better.

\section{Conclusion}
We presented SVAG, a variance-reduced stochastic gradient method with adjustable bias and with SAG and SAGA as special cases.
It was analyzed in two scenarios, one being the minimization problem of a finite sum of functions with cocoercive gradients and the other being finding a root of a finite sum of cocoercive operators.
The analysis improves on the previously best known analyses in both settings and, more significantly, the two different scenarios gave significantly different convergence conditions for the step-size.
In the cocoercive operator setting a much more restrictive condition was found and it was verified numerically.
This difference is not present in ordinary gradient descent and can therefore easily be overlooked, however, these results suggest that is inadvisable in the variance-reduced stochastic gradient setting.

The theoretical results in the minimization case was further examined with numerical experiments.
Several choices of bias was examined but we did not find the same dependence on the bias that the theory suggests.
In fact, the asymptotic convergence behavior was similar for the different choices of bias, indicating that further improvements of the theory is still needed.
The bias mainly impacted the early stages of the convergence and in a couple of cases this impact was significant.
There might therefore still be benefits to tuning the bias to the particular problem but further work is needed to efficiently do so.

\paragraph{\normalfont{\textbf{Funding}}}
This work is funded by the Swedish Research Council via grant number 2016-04646.

\bibliography{references}

\newpage
{
	\textbf{\LARGE\centering
		Supplementary Proofs \\
		\rule{0.6\linewidth}{1pt}\vspace{0.5em} \\
		Cocoercivity, Smoothness and Bias in Variance-Reduced Stochastic Gradient Methods \\
	}
}

\appendix
\section{Matrix Identities}
Here we verify the matrix identities used in the proofs.

\subsection{}
\begin{align*}
	\Exp[U_{i^k}] =
	\frac{1}{n}\sum_{i=1}^n\begin{bmatrix}
		\frac{\lambda}{n}\theta e_{i^k}^T & -\frac{\lambda}{n}\theta e_{i^k}^T+\frac{\lambda}{n}\mathbf{1}^T \\
		-E_{i^k} & E_{i^k}
	\end{bmatrix}=
	\frac{1}{n}
	\begin{bmatrix}
		\frac{\lambda}{n}\theta\mathbf{1}^T & \frac{\lambda}{n}(n-\theta)\mathbf{1}^T \\
		-I & I
	\end{bmatrix}
\end{align*}

\subsection{}
\begin{align*}
	H\Exp[U_{i^k}]
	&=
	\frac{1}{n}
	\begin{bmatrix}
		1 & - \frac{\lambda}{n}(n-\theta) \mathbf{1}^T \\
		-\frac{\lambda}{n}(n-\theta)\mathbf{1} & \frac{\lambda}{L} I + \frac{\lambda^2}{n^2}(n-\theta)^2 \mathbf{11}^T
	\end{bmatrix}
	\begin{bmatrix}
		\frac{\lambda}{n}\theta\mathbf{1}^T & \frac{\lambda}{n}(n-\theta)\mathbf{1}^T \\
		-I & I
	\end{bmatrix}\\
	&=
	\begin{bmatrix}
		\frac{\lambda}{n}\mathbf{1}^T & 0 \\
		- \frac{\lambda^2}{n^2}(n-\theta)\mathbf{11}^2 - \frac{\lambda}{nL}I & \frac{\lambda}{nL}I
	\end{bmatrix}
\end{align*}

\subsection{}
For $\Exp[U_{i^k}^T H U_{i^k}]$, first note that $U_{i^k}$ can be written as
\begin{align*}
	U_{i^k} &=
	\begin{bmatrix}
		\frac{\lambda}{n}\theta e_{i^k}^T & -\frac{\lambda}{n}\theta e_{i^k}^T+\frac{\lambda}{n}\mathbf{1}^T \\
		-E_{i^k} & E_{i^k}
	\end{bmatrix}
	=
	\begin{bmatrix}
		e_{i^k}^T & 0 \\
		0 & E_{i^k}
	\end{bmatrix}
	\begin{bmatrix}
		\frac{\lambda}{n}\theta I & -\frac{\lambda}{n}\theta I \\
		-I & I
	\end{bmatrix} +
	\begin{bmatrix}
		0 & \frac{\lambda}{n} \mathbf{1}^T \\
		0 & 0
	\end{bmatrix},
\end{align*}
which gives
\begin{align*}
	&U_{i^k}^T H U_{i^k} \\
	&\quad=
	\left(
		\begin{bmatrix}
			\frac{\lambda}{n}\theta I & -I \\
			-\frac{\lambda}{n}\theta I & I
		\end{bmatrix}
		\begin{bmatrix}
			e_{i^k} & 0 \\
			0 & E_{i^k}
		\end{bmatrix}
		+
		\begin{bmatrix}
			0 & 0 \\
			\frac{\lambda}{n} \mathbf{1} & 0
		\end{bmatrix}
	\right)
	H
	\left(
		\begin{bmatrix}
			e_{i^k}^T & 0 \\
			0 & E_{i^k}
		\end{bmatrix}
		\begin{bmatrix}
			\frac{\lambda}{n}\theta I & -\frac{\lambda}{n}\theta I \\
			-I & I
		\end{bmatrix}
		+
		\begin{bmatrix}
			0 & \frac{\lambda}{n}\mathbf{1}^T \\
			0 & 0
		\end{bmatrix}
	\right)
	\\
	&\quad=
	\begin{bmatrix}
		0 & 0 \\
		\frac{\lambda}{n} \mathbf{1} & 0
	\end{bmatrix}
	H
	\begin{bmatrix}
		0 & \frac{\lambda}{n} \mathbf{1}^T \\
		0 & 0
	\end{bmatrix}
	+
	\begin{bmatrix}
		\frac{\lambda}{n}\theta I & -I \\
		-\frac{\lambda}{n}\theta I & I
	\end{bmatrix}
	\begin{bmatrix}
		e_{i^k} & 0 \\
		0 & E_{i^k}
	\end{bmatrix}
	H
	\begin{bmatrix}
		e_{i^k}^T & 0 \\
		0 & E_{i^k}
	\end{bmatrix}
	\begin{bmatrix}
		\frac{\lambda}{n}\theta I & -\frac{\lambda}{n}\theta I \\
		-I & I
	\end{bmatrix}
	\\
	&\quad\quad
	+
	\begin{bmatrix}
		0 & 0 \\
		\frac{\lambda}{n} \mathbf{1} & 0
	\end{bmatrix}
	H
	\begin{bmatrix}
		e_{i^k}^T & 0 \\
		0 & E_{i^k}
	\end{bmatrix}
	\begin{bmatrix}
		\frac{\lambda}{n}\theta I & - \frac{\lambda}{n}\theta I \\
		-I & I
	\end{bmatrix}
	+
	\begin{bmatrix}
		\frac{\lambda}{n}\theta I & -I \\
		-\frac{\lambda}{n}\theta I & I
	\end{bmatrix}
	\begin{bmatrix}
		e_{i^k} & 0 \\
		0 & E_{i^k}
	\end{bmatrix}
	H
	\begin{bmatrix}
		0 & \frac{\lambda}{n} \mathbf{1}^T \\
		0 & 0
	\end{bmatrix}
	.
\end{align*}
The first term is
\begin{align*}
	&
	\begin{bmatrix}
		0 & 0 \\
		\frac{\lambda}{n} \mathbf{1} & 0
	\end{bmatrix}
	H
	\begin{bmatrix}
		0 & \frac{\lambda}{n} \mathbf{1}^T \\
		0 & 0
	\end{bmatrix} \\
	&\quad=
	\begin{bmatrix}
		0 & 0 \\
		\frac{\lambda}{n} \mathbf{1} & 0
	\end{bmatrix}
	\begin{bmatrix}
		1 & - \frac{\lambda}{n}(n-\theta) \mathbf{1}^T \\
		-\frac{\lambda}{n}(n-\theta)\mathbf{1} & \frac{\lambda}{L} I + \frac{\lambda^2}{n^2}(n-\theta)^2 \mathbf{11}^T
	\end{bmatrix}
	\begin{bmatrix}
		0 & \frac{\lambda}{n} \mathbf{1}^T \\
		0 & 0
	\end{bmatrix}\\
	&\quad=
	\begin{bmatrix}
		0 & 0 \\
		\frac{\lambda}{n}\mathbf{1} & - \frac{\lambda^2}{n^2}(n-\theta)\mathbf{11}^T
	\end{bmatrix}
	\begin{bmatrix}
		0 & \frac{\lambda}{n} \mathbf{1}^T \\
		0 & 0
	\end{bmatrix}\\
	&\quad=
	\begin{bmatrix}
		0 & 0 \\
		0 & \frac{\lambda^2}{n^2}\mathbf{11}^T
	\end{bmatrix}.
\end{align*}
The second term is
\begin{align*}
	&
	\begin{bmatrix}
		\frac{\lambda}{n}\theta I & -I \\
		-\frac{\lambda}{n}\theta I & I
	\end{bmatrix}
	\begin{bmatrix}
		e_{i^k} & 0 \\
		0 & E_{i^k}
	\end{bmatrix}
	H
	\begin{bmatrix}
		e_{i^k}^T & 0 \\
		0 & E_{i^k}
	\end{bmatrix}
	\begin{bmatrix}
		\frac{\lambda}{n}\theta I & -\frac{\lambda}{n}\theta I \\
		-I & I
	\end{bmatrix}
	\\
	&\quad=
	\begin{bmatrix}
		\frac{\lambda}{n}\theta I & -I \\
		-\frac{\lambda}{n}\theta I & I
	\end{bmatrix}
	\begin{bmatrix}
		e_{i^k} & 0 \\
		0 & E_{i^k}
	\end{bmatrix}
	\begin{bmatrix}
		1 & - \frac{\lambda}{n}(n-\theta) \mathbf{1}^T \\
		-\frac{\lambda}{n}(n-\theta)\mathbf{1} & \frac{\lambda}{L} I + \frac{\lambda^2}{n^2}(n-\theta)^2 \mathbf{11}^T
	\end{bmatrix}
	\begin{bmatrix}
		e_{i^k}^T & 0 \\
		0 & E_{i^k}
	\end{bmatrix}
	\begin{bmatrix}
		\frac{\lambda}{n}\theta I & -\frac{\lambda}{n}\theta I \\
		-I & I
	\end{bmatrix}
	\\
	&\quad=
	\begin{bmatrix}
		\frac{\lambda}{n}\theta I & -I \\
		-\frac{\lambda}{n}\theta I & I
	\end{bmatrix}
	\begin{bmatrix}
		e_{i^k} & 0 \\
		0 & E_{i^k}
	\end{bmatrix}
	\begin{bmatrix}
		e_{i^k}^T & -\frac{\lambda}{n}(n-\theta)e_{i^k}^T \\
		-\frac{\lambda}{n}(n-\theta)\mathbf{1}e_{i^k}^T  & \frac{\lambda}{L}E_{i^k} + \frac{\lambda^2}{n^2}(n-\theta)^2 \mathbf{1}e_{i^k}^T
	\end{bmatrix}
	\begin{bmatrix}
		\frac{\lambda}{n}\theta I & -\frac{\lambda}{n}\theta I \\
		-I & I
	\end{bmatrix}
	\\
	&\quad=
	\begin{bmatrix}
		\frac{\lambda}{n}\theta I & -I \\
		-\frac{\lambda}{n}\theta I & I
	\end{bmatrix}
	\begin{bmatrix}
		E_{i^k} & -\frac{\lambda}{n}(n-\theta)E_{i^k} \\
		-\frac{\lambda}{n}(n-\theta)E_{i^k} & \frac{\lambda}{L}E_{i^k} + \frac{\lambda^2}{n^2}(n-\theta)^2E_{i^k}
	\end{bmatrix}
	\begin{bmatrix}
		\frac{\lambda}{n}\theta I & -\frac{\lambda}{n}\theta I \\
		-I & I
	\end{bmatrix}
	\\
	&\quad=
	\begin{bmatrix}
		\frac{\lambda}{n}\theta I & -I \\
		-\frac{\lambda}{n}\theta I & I
	\end{bmatrix}
	\begin{bmatrix}
		\lambda E_{i^k} & -\lambda E_{i^k} \\
		- \frac{\lambda}{L}E_{i^k} - \frac{\lambda^2}{n}(n-\theta)E_{i^k} & \frac{\lambda}{L}E_{i^k} + \frac{\lambda^2}{n}(n-\theta)E_{i^k}
	\end{bmatrix}
	\\
	&\quad=
	\begin{bmatrix}
		\frac{\lambda}{L}E_{i^k} + \lambda^2E_{i^k} & -\frac{\lambda}{L}E_{i^k} - \lambda^2E_{i^k}  \\
		-\frac{\lambda}{L}E_{i^k} - \lambda^2E_{i^k}  & \frac{\lambda}{L}E_{i^k} + \lambda^2E_{i^k}
	\end{bmatrix} \\
	&\quad=
	\begin{bmatrix}
		1 & -1 \\
		-1 & 1
	\end{bmatrix} \otimes (\frac{\lambda}{L} + \lambda^2) E_{i^k}.
\end{align*}
The third and forth term are
\begin{align*}
	&
	\begin{bmatrix}
		0 & 0 \\
		\frac{\lambda}{n} \mathbf{1} & 0
	\end{bmatrix}
	H
	\begin{bmatrix}
		e_{i^k}^T & 0 \\
		0 & E_{i^k}
	\end{bmatrix}
	\begin{bmatrix}
		\frac{\lambda}{n}\theta I & - \frac{\lambda}{n}\theta I \\
		-I & I
	\end{bmatrix}
	\\
	&\quad=
	\begin{bmatrix}
		0 & 0 \\
		\frac{\lambda}{n} \mathbf{1} & 0
	\end{bmatrix}
	\begin{bmatrix}
		1 & - \frac{\lambda}{n}(n-\theta) \mathbf{1}^T \\
		-\frac{\lambda}{n}(n-\theta)\mathbf{1} & \frac{\lambda}{L} I + \frac{\lambda^2}{n^2}(n-\theta)^2 \mathbf{11}^T
	\end{bmatrix}
	\begin{bmatrix}
		e_{i^k}^T & 0 \\
		0 & E_{i^k}
	\end{bmatrix}
	\begin{bmatrix}
		\frac{\lambda}{n}\theta I & - \frac{\lambda}{n}\theta I \\
		-I & I
	\end{bmatrix}
	\\
	&\quad=
	\begin{bmatrix}
		0 & 0 \\
		\frac{\lambda}{n}\mathbf{1} & -\frac{\lambda^2}{n^2}(n-\theta)\mathbf{11}^T
	\end{bmatrix}
	\begin{bmatrix}
		e_{i^k}^T & 0 \\
		0 & E_{i^k}
	\end{bmatrix}
	\begin{bmatrix}
		\frac{\lambda}{n}\theta I & - \frac{\lambda}{n}\theta I \\
		-I & I
	\end{bmatrix}
	\\
	&\quad=
	\begin{bmatrix}
		0 & 0 \\
		\frac{\lambda}{n}\mathbf{1}e_{i^k}^T & -\frac{\lambda^2}{n^2}(n-\theta)\mathbf{1} e_{i^k}^T
	\end{bmatrix}
	\begin{bmatrix}
		\frac{\lambda}{n}\theta I & -\frac{\lambda}{n}\theta I \\
		-I & I
	\end{bmatrix}
	\\
	&\quad=
	\begin{bmatrix}
		0 & 0 \\
		\frac{\lambda^2}{n}\mathbf{1} e_{i^k}^T & -\frac{\lambda^2}{n}\mathbf{1} e_{i^k}^T
	\end{bmatrix}.
\end{align*}
This results in
\begin{align*}
	U_{i^k}^T H U_{i^k}
	&=
	\begin{bmatrix}
		0 & 0 \\
		0 & \frac{\lambda^2}{n^2}\mathbf{11}^T
	\end{bmatrix}
	+
	\begin{bmatrix}
		1 & -1 \\
		-1 & 1
	\end{bmatrix} \otimes (\frac{\lambda}{L} + \lambda^2) E_{i^k}
	+
	\frac{\lambda^2}{n}\begin{bmatrix}
		0 & e_{i^k}\mathbf{1}^T \\
		\mathbf{1} e_{i^k}^T & -\mathbf{1} e_{i^k}^T - e_{i^k}\mathbf{1}^T
	\end{bmatrix}
\end{align*}
and
\begin{align*}
	\Exp[U_{i^k}^T H U_{i^k}]
	&=
	\begin{bmatrix}
		1 & -1 \\
		-1 & 1
	\end{bmatrix} \otimes (\frac{1}{L} + \lambda)\frac{\lambda}{n} I
	+
	\begin{bmatrix}
		0 & 1 \\
		1 & -1
	\end{bmatrix} \otimes \frac{\lambda^2}{n^2}\mathbf{11}^T.
\end{align*}

\subsection{}
\begin{align*}
	&2M -  \Exp[U_{i^k}^T H U_{i^k}] - \xi I
	\\
	&\quad=
	\begin{bmatrix}
		2 & -1 \\
		-1 & 2
	\end{bmatrix} \otimes \frac{\lambda}{nL} I
	-
	\begin{bmatrix}
		0 & 1\\
		1 & 0
	\end{bmatrix} \otimes (n-\theta)\frac{\lambda^2}{n^2}\mathbf{11}^T
	\\
	&\quad -
	\begin{bmatrix}
		1 & -1 \\
		-1 & 1
	\end{bmatrix} \otimes (\frac{1}{L} + \lambda)\frac{\lambda}{n} I
	-
	\begin{bmatrix}
		0 & 1 \\
		1 & -1
	\end{bmatrix} \otimes \frac{\lambda^2}{n^2}\mathbf{11}^T
	-
	\begin{bmatrix}
		1 & 0\\
		0 & 1
	\end{bmatrix} \otimes \xi I
	\\
	&\quad=
	\begin{bmatrix}
		1 & 0 \\
		0 & 1
	\end{bmatrix} \otimes \frac{\lambda}{nL} I
	-
	\begin{bmatrix}
		1 & 0 \\
		0 & 1
	\end{bmatrix} \otimes \frac{\lambda^2}{n} I
	+
	\begin{bmatrix}
		0 & 1 \\
		1 & 0
	\end{bmatrix} \otimes \frac{\lambda^2}{n}(I - \frac{1}{n}\mathbf{11}^T)
	\\
	&\quad
	-
	\begin{bmatrix}
		0 & 1\\
		1 & 0
	\end{bmatrix} \otimes (n-\theta)\frac{\lambda^2}{n^2}\mathbf{11}^T
	+
	\begin{bmatrix}
		0 & 0 \\
		0 & 1
	\end{bmatrix} \otimes \frac{\lambda^2}{n^2}\mathbf{11}^T
	-
	\begin{bmatrix}
		1 & 0\\
		0 & 1
	\end{bmatrix} \otimes \xi I
\end{align*}

\subsection{}
We have $S = S_L + S_C$ where
\begin{align*}
	S_L
	&= \Exp[Q_{i^k}^T G + G^T Q_{i^k} - L Q_{i^k}^T Q_{i^k}] \\
	&= \Exp\Bigg[
		\frac{\lambda}{n}
		\begin{bmatrix}
			(\theta - 1)e_{i^k} \\ -(\theta - 1)e_{i^k}
		\end{bmatrix}
		\frac{1}{n}
		\begin{bmatrix}
			\mathbf{1}^T & 0
		\end{bmatrix}
		+ G^T Q_{i^k} - L Q_{i^k}^T Q_{i^k}
		\Bigg|
		\mathcal{F}_k
	\Bigg] \\
	&= \Exp\Bigg[
		\frac{\lambda}{n^2}
		\begin{bmatrix}
			(\theta - 1)e_{i^k}\mathbf{1}^T & 0\\
			-(\theta - 1)e_{i^k}\mathbf{1}^T & 0
		\end{bmatrix}
		+ G^T Q_{i^k} - L Q_{i^k}^T Q_{i^k}
		\Bigg|
		\mathcal{F}_k
	\Bigg] \\
	&=
	\frac{\lambda}{n^3}
	\begin{bmatrix}
		(\theta - 1)\mathbf{11}^T & 0\\
		-(\theta - 1)\mathbf{11}^T & 0
	\end{bmatrix}
	+ \Exp[G^T Q_{i^k} - L Q_{i^k}^T Q_{i^k}] \\
	&=
	\frac{\lambda}{n^3}
	\begin{bmatrix}
		2(\theta - 1)\mathbf{11}^T & -(\theta - 1)\mathbf{11}^T\\
		-(\theta - 1)\mathbf{11}^T & 0
	\end{bmatrix}
	- \Exp[L Q_{i^k}^T Q_{i^k} ]\\
	&=
	\frac{\lambda}{n^3}
	\begin{bmatrix}
		2(\theta - 1)\mathbf{11}^T & -(\theta - 1)\mathbf{11}^T\\
		-(\theta - 1)\mathbf{11}^T & 0
	\end{bmatrix}
	\\
	&\quad-
	\Exp
	\Bigg[
		\frac{L\lambda^2}{n^2}
		\begin{bmatrix}
			(\theta - 1)e_{i^k} \\ -(\theta - 1)e_{i^k}
		\end{bmatrix}
		\begin{bmatrix}
			(\theta - 1)e_{i^k}^T & -(\theta - 1)e_{i^k}^T
		\end{bmatrix}
		\Bigg|
		\mathcal{F}_k
	\Bigg]
	\\
	&=
	\frac{\lambda}{n^3}
	\begin{bmatrix}
		2(\theta - 1)\mathbf{11}^T & -(\theta - 1)\mathbf{11}^T\\
		-(\theta - 1)\mathbf{11}^T & 0
	\end{bmatrix}
	-
	\Exp
	\Bigg[
		\frac{L\lambda^2}{n^2}
		\begin{bmatrix}
			(\theta - 1)^2 E_{i^k} & -(\theta - 1)^2 E_{i^k}  \\
			-(\theta - 1)^2 E_{i^k} & (\theta - 1)^2 E_{i^k}
		\end{bmatrix}
		\Bigg|
		\mathcal{F}_k
	\Bigg]
	\\
	&=
	\begin{bmatrix}
		2 & -1 \\
		-1 & 0
	\end{bmatrix} \otimes (\theta - 1)\frac{\lambda}{n^3} \mathbf{11}^T
	-
	\begin{bmatrix}
		1  & -1  \\
		-1  & 1
	\end{bmatrix} \otimes (\theta - 1)^2\frac{L\lambda^2}{n^3}  I
\end{align*}
and
\begin{align*}
	S_C
	&= \frac{\lambda}{n}(D^TG + G^TD) \\
	&= \frac{\lambda}{n}
	\Bigg[
		\begin{bmatrix} 0 \\ \mathbf{1} \end{bmatrix}
		\begin{bmatrix} \frac{1}{n}\mathbf{1}^T & 0 \end{bmatrix}
		+ G^TD
	\Bigg] \\
	&=
	\frac{\lambda}{n}
	\Bigg[
		\begin{bmatrix}
			0 & 0 \\
			\frac{1}{n}\mathbf{11}^T & 0
		\end{bmatrix}
		+ G^TD
	\Bigg] \\
	&=
	\begin{bmatrix}
		0 & 1 \\
		1 & 0
	\end{bmatrix} \otimes \frac{\lambda}{n^2} \mathbf{11}^T.
\end{align*}

\subsection{}
\begin{align*}
	&2M - \Exp[U_{i^k}^T H U_{i^k}] + \lambda (n-\theta) S - \xi I\\
	&\quad=
	\begin{bmatrix}
		1 & 0 \\
		0 & 1
	\end{bmatrix} \otimes \frac{\lambda}{nL} I
	-
	\begin{bmatrix}
		1 & 0 \\
		0 & 1
	\end{bmatrix} \otimes \frac{\lambda^2}{n} I
	+
	\begin{bmatrix}
		0 & 1 \\
		1 & 0
	\end{bmatrix} \otimes \frac{\lambda^2}{n}(I - \frac{1}{n}\mathbf{11}^T)
	\\
	&\quad\quad
	-
	\begin{bmatrix}
		0 & 1\\
		1 & 0
	\end{bmatrix} \otimes (n-\theta)\frac{\lambda^2}{n^2}\mathbf{11}^T
	+
	\begin{bmatrix}
		0 & 0 \\
		0 & 1
	\end{bmatrix} \otimes \frac{\lambda^2}{n^2}\mathbf{11}^T
	\\
	&\quad\quad +
	\lambda (n-\theta)
	\Bigg(
		\begin{bmatrix}
			2 & -1 \\
			-1 & 0
		\end{bmatrix} \otimes (\theta - 1)\frac{\lambda}{n^3} \mathbf{11}^T
		-
		\begin{bmatrix}
			1  & -1  \\
			-1  & 1
		\end{bmatrix} \otimes (\theta - 1)^2\frac{L\lambda^2}{n^3}  I \\
		&\quad\quad\quad +
		\begin{bmatrix}
			0 & 1 \\
			1 & 0
		\end{bmatrix} \otimes \frac{\lambda}{n^2} \mathbf{11}^T
	\Bigg)
	\\
	&\quad\quad-
	\begin{bmatrix}
		1 & 0\\
		0 & 1
	\end{bmatrix} \otimes \xi I
	\\
	\\
	&\quad=
	\begin{bmatrix}
		1 & 0 \\
		0 & 1
	\end{bmatrix} \otimes \frac{\lambda}{nL} I
	-
	\begin{bmatrix}
		1 & 0 \\
		0 & 1
	\end{bmatrix} \otimes \frac{\lambda^2}{n} I
	+
	\begin{bmatrix}
		0 & 1 \\
		1 & 0
	\end{bmatrix} \otimes \frac{\lambda^2}{n}(I - \frac{1}{n}\mathbf{11}^T)
	\\
	&\quad\quad
	-
	\begin{bmatrix}
		0 & 1\\
		1 & 0
	\end{bmatrix} \otimes (n-\theta)\frac{\lambda^2}{n^2}\mathbf{11}^T
	+
	\begin{bmatrix}
		0 & 0 \\
		0 & 1
	\end{bmatrix} \otimes \frac{\lambda^2}{n^2}\mathbf{11}^T
	\\
	&\quad\quad +
	\begin{bmatrix}
		2 & -1 \\
		-1 & 0
	\end{bmatrix} \otimes  (n-\theta)(\theta - 1)\frac{\lambda^2}{n^3} \mathbf{11}^T
	-
	\begin{bmatrix}
		1  & -1  \\
		-1  & 1
	\end{bmatrix} \otimes  (n-\theta)(\theta - 1)^2\frac{L\lambda^3}{n^3}  I
	\\
	&\quad\quad+
	\begin{bmatrix}
		0 & 1 \\
		1 & 0
	\end{bmatrix} \otimes (n-\theta)\frac{\lambda^2}{n^2} \mathbf{11}^T
	\\
	&\quad\quad-
	\begin{bmatrix}
		1 & 0\\
		0 & 1
	\end{bmatrix} \otimes \xi I
	\\
	\\
	&\quad=
	\begin{bmatrix}
		1 & 0 \\
		0 & 1
	\end{bmatrix} \otimes \frac{\lambda}{nL} I
	-
	\begin{bmatrix}
		1 & 0 \\
		0 & 1
	\end{bmatrix} \otimes \frac{\lambda^2}{n} I
	+
	\begin{bmatrix}
		0 & 1 \\
		1 & 0
	\end{bmatrix} \otimes \frac{\lambda^2}{n}(I - \frac{1}{n}\mathbf{11}^T)
	+
	\begin{bmatrix}
		0 & 0 \\
		0 & 1
	\end{bmatrix} \otimes \frac{\lambda^2}{n^2}\mathbf{11}^T
	\\
	&\quad\quad
	\\
	&\quad\quad +
	\begin{bmatrix}
		2 & -1 \\
		-1 & 0
	\end{bmatrix} \otimes  (n-\theta)(\theta - 1)\frac{\lambda^2}{n^3} \mathbf{11}^T
	-
	\begin{bmatrix}
		1  & -1  \\
		-1  & 1
	\end{bmatrix} \otimes  (n-\theta)(\theta - 1)^2\frac{L\lambda^3}{n^3}  I
	\\
	&\quad\quad-
	\begin{bmatrix}
		1 & 0\\
		0 & 1
	\end{bmatrix} \otimes \xi I
\end{align*}

\end{document}